\documentclass[USenglish, a4paper, 10pt]{article}
\usepackage[utf8]{inputenc} 
\usepackage{babel} 
\usepackage{lmodern} 
\usepackage[T1]{fontenc} 
\usepackage{amssymb, amsmath, mathrsfs, mathabx, bbm} 
\DeclareMathAlphabet{\mathbbold}{U}{bbold}{m}{n}
\usepackage{stmaryrd}
\usepackage[a4paper]{geometry} 
\usepackage{titling} 
\usepackage{titlesec} 
		\titlelabel{\thetitle.\quad} 
		\titleformat*{\section}{\center\large} 
		\titleformat*{\subsection}{\sf\large} 
		\titleformat*{\subsubsection}{\sf\it} 
		
\usepackage[colorlinks=true,linkcolor=black, citecolor=GreenCite,urlcolor=GreenCite]{hyperref} 
\usepackage[usenames,dvipsnames]{color} 
\definecolor{ToDo}{RGB}{30,144,255}
\definecolor{Provisional}{RGB}{218,165,32}
\definecolor{Question}{RGB}{220,20,60}
\definecolor{GreenCite}{RGB}{47, 79, 79}

\usepackage{verbatim} 
\usepackage{array} 
\usepackage[all]{xy} 
\usepackage[usenames,dvipsnames]{color} 
\usepackage{enumerate} 
\usepackage{leftidx}
\usepackage{paralist}

\numberwithin{equation}{section} 
\usepackage{amsthm} 
\swapnumbers 
\theoremstyle{plain}
\newtheorem{theoSec}{Theorem}[section] 

\newtheorem{lemSec}[theoSec]{Lemma}
\newtheorem{proSec}[theoSec]{Proposition}
\newtheorem{corSec}[theoSec]{Corollary}
\newtheorem{defiSec}[theoSec]{Definition}
\newtheorem{remSec}[theoSec]{Remark}

\newtheorem{exsSec}[theoSec]{Examples}
\theoremstyle{remark}

\theoremstyle{plain}
\newtheorem{theo}{Theorem}[subsection] 

\newtheorem{lem}[theo]{Lemma}
\newtheorem{pro}[theo]{Proposition}

\newtheorem{defi}[theo]{Definition}
\newtheorem{rem}[theo]{Remark}

\newtheorem{ex}[theo]{Example}

\theoremstyle{remark}
\newtheorem{note}[theo]{Note}

\theoremstyle{plain}

\theoremstyle{remark}

\title{\textbf{Quantum direct products and the Künneth class}}
\author{\textsc{Rubén Martos}\thanks{Department of Mathematical Sciences, University of Copenhagen (Denmark). R.M. is supported by the European Union's Horizon 2020 research and innovation programme under the Marie Skłodowska-Curie grant agreement No 895141.}}
\date{}
\begin{document}
\maketitle
\renewcommand{\abstractname}{}
\vspace{-2.5cm}
\begin{abstract}
\textsc{Abstract}. We introduce a Künneth class in the quantum equivariant setting inspired by the pioneer work by J. Chabert, H. Oyono-Oyono and S. Echterhoff, which allows to relate the quantum Baum-Connes property with the Künneth formula by generalising some key results of Chabert-Oyono-Oyono-Echterhoff to discrete quantum groups. Finally, we make the observation that the C$^*$-algebra defining a compact quantum group with dual satisfying the strong quantum Baum-Connes property belongs to the Künneth class. This allows to obtain some K-theory computations for quantum direct products based on earlier work by Voigt and Vergnioux-Voigt.
	
\end{abstract}

\tableofcontents

\section{\textsc{Introduction}}
	
	Universal Coefficient Theorems in algebraic topology establish a connection between ordinary homology (resp. cohomology) with homology (resp. cohomology) with coefficients. In noncommutative geometry, C$^*$-algebras are viewed as noncommutative analogues of topological spaces and as such it is reasonable to extend ideas from algebraic topology to a noncommutative framework. In this sense, K-theory of C$^*$-algebras is viewed as a homology theory for C$^*$-algebras, which provides important invariants for C$^*$-algebras (e.g. AF-algebras are completely classified through K-theory \cite{ElliotAF}). Dually, K-homology for C$^*$-algebras is a cohomology theory for C$^*$-algebras, which agrees with the $\text{Ext}$-functor in the commutative unital case. K-theory and K-homology can be related to each other by means of the index theory of elliptic pseudo-differential operators (e.g. the Atiyah-Singer's theorem). Kasparov KK-theory is in turn a bivariant K-theory, which gathers together both K-theory and K-homology of C$^*$-algebras. KK-theory plays a central role in the classification program of C$^*$-algebras, but it is also relevant to approach diverse problems outside noncommutative geometry, e.g. the Novikov conjecture.
	
	The Universal Coefficient Theorem (\emph{UCT} for short) of J. Rosenberg and C. Schochet \cite{RosenbergSchochet} approximates the bivariant K-theory of Kasparov in terms of ordinary K-theory. More precisely, given two (separable) C$^*$-algebras $A$ and $B$, there exists a short exact sequence:
	$$\text{Ext}^1_{\mathbb{Z}}(K_*(A), K_*(B))\rightarrowtail KK(A, B)\twoheadrightarrow \text{Hom}(K_*(A), K_*(B)),$$
	provided that $A$ belongs to certain bootstrap class. A remarkable consequence of this is that K-equivalences between C$^*$-algebras satisfying the UCT lift to KK-equivalences. A classical companion to the UCT is the \emph{Künneth theorem} (cf. \cite{RosenbergSchochet}), which asserts that given two (separable) C$^*$-algebras $A$ and $B$, there exists a short exact sequence:
	$$K_*(A)\otimes K_*(B)\rightarrowtail K_*(A\otimes B)\twoheadrightarrow \text{Tor}^{\mathbb{Z}}_1(K_*(A), K_*(B)),$$
	provided that $A$ belongs to certain bootstrap class. Roughly speaking, Künneth theorem allows to compute the K-theory of a tensor product of two C$^*$-algebras in terms of the K-theory of the corresponding factors (e.g. when $K(A)$ or $K(B)$ is torsion-free). When the the first arrow in the above diagram is an isomorphism, we refer to $K_*(A)\otimes K_*(B)\cong K_*(A\otimes B)$ as the \emph{Künneth formula}. Classically, the Künneth formula identifies the (co-)homology of a product of topological spaces with the tensor product of (co-)homologies.
	
	In \cite{ChabertEchterhoffOyono}, J. Chabert, S. Echterhoff and H. Oyono-Oyono establish a connection between the Baum-Connes conjecture for a locally compact group $G$ with coefficients in a C$^*$-algebra $A$ and the Künneth formula for the K-theory of tensor products by the corresponding crossed product $A\underset{r}{\rtimes} G$. One remarkable application of the techniques developed by J. Chabert, S. Echterhoff and H. Oyono-Oyono is a permanence property of the Baum-Connes conjecture for a direct product of locally compact groups with trivial coefficients.
	
	In this article, we establish such a connection in analogy to \cite{ChabertEchterhoffOyono} in the realm of discrete quantum groups and prove a similar permanence property for the quantum counterpart of the Baum-Connes conjecture. It is important to say that such a study was partially initiated by the author in \cite{RubenSemiDirect}, but only under torsion-freeness assumption. This was due mainly to two technical difficulties: the classification of torsion actions for quantum direct products and the formulation of a quantum assembly map for arbitrary discrete quantum groups. Let us give more details about these issues.
	
	The Baum-Connes conjecture has been formulated in 1982 by P. Baum and A. Connes. We still do not know any counter example to the original conjecture but it is known that the one with coefficients is false. For this reason we refer to the Baum-Connes conjecture with coefficients as the \emph{Baum-Connes property} (\emph{BC property} for short). The conjecture aims to understand the link between two operator K-groups of different nature, which would establish a strong connection between geometry and topology in a more abstract and general index-theory context. More precisely, if $G$ is a (second countable) locally compact group and $A$ is a (separable) $G$-$C^*$-algebra, then the BC property for $G$ with coefficients in $A$ claims that the assembly map $\mu^G_A: K^{top}_{*}(G; A)\longrightarrow K_{*}(A\underset{r}{\rtimes}G)$ is an isomorphism, where $K^{top}_{*}(G; A)$ is the equivariant K-homology with compact support of $G$ with coefficients in $A$ and $K_{*}(A\underset{r}{\rtimes}G)$  is the K-theory of the reduced crossed product $A\underset{r}{\rtimes}G$. This property has been proved for a large class of groups (e.g. a-T-menable groups \cite{HigsonKasparovHaagerup} or hyperbolic groups \cite{Lafforgue}).
	
	On the one hand, a major problem when studying the quantum counterpart of the BC property for a discrete quantum group $\widehat{\mathbb{G}}$ is the torsion structure of $\widehat{\mathbb{G}}$. Indeed, if $G$ is a discrete group, its torsion phenomenon is completely described in terms of the finite subgroups of $G$ and encoded in the localizing subcategory of $\mathscr{KK}^{G}$ of \emph{compactly induced $G$-$C^*$-algebras}, denoted by $\mathscr{L}_G$, according to the Meyer-Nest reformulation \cite{MeyerNest}. We say that $G$ satisfies the \emph{strong} BC property if $\mathscr{L}_G=\mathscr{KK}^{G}$, which corresponds, in usual terms, to the existence of a $\gamma$-element that equals $\mathbbold{1}_{\mathbb{C}}$ (cf. \cite{MeyerNest}). This approach yields as well a characterization of the BC property for a discrete group $G$ \emph{only} in terms of finite subgroups, K-theory and crossed products. The notion of torsion for a genuine discrete quantum group, $\widehat{\mathbb{G}}$, has been introduced firstly by R. Meyer and R. Nest \cite{MeyerNestTorsion}, \cite{MeyerNestHomological2} in terms of ergodic actions of $\mathbb{G}$. It has been re-interpreted later by Y. Arano and K. De Commer in terms of fusion rings and module $C^*$-categories \cite{YukiKenny}.
	
	Given a compact quantum group $\mathbb{G}$, an important question in this context is to classify all torsion actions of $\mathbb{G}$ (at least up to equivariant Morita equivalence). Such a classification is known for the most common examples of compact quantum groups, e.g for the q-deformation of $SU(2)$, the free unitary quantum group $U^+(n)$, the free orthogonal quantum group $O^+(n)$ or the quantum permutation group $S^+_N$ with $n, N\in\mathbb{N}$ (cf. \cite{VoigtBaumConnesFree, VoigtBaumConnesAutomorphisms, YukiKenny}). In fact, $\widehat{SU_q(2)}$, $\widehat{U^+(n)}$ and $\widehat{O^+(n)}$ are all torsion-free and the only, up to equivariant Morita equivalence, non-trivial torsion action of $S^+_N$ is its structural action as quantum automorphism group of $\mathbb{C}^N$. However, little is known about the classification of torsion actions for general constructions of quantum groups. In this direction, the only, to the best knowledge to the author, such a general result was given in \cite{RubenAmauryTorsion} by the author in collaboration with A. Freslon. It is shown in \cite{RubenAmauryTorsion} that torsion actions of a free product of compact quantum groups are in one-to-one correspondence, up to equivariant Morita equivalence, with torsion actions of each of the factors. 
	
	It is important to mention that by virtue of the work \cite{YukiKenny} by Y. Arano and K. De Commer, we know that if both $\widehat{\mathbb{G}}$ and $\widehat{\mathbb{H}}$ are torsion-free, then $\widehat{\mathbb{G}\times \mathbb{H}}$ is torsion-free too. The converse is also true because both $\widehat{\mathbb{G}}$ and $\widehat{\mathbb{H}}$ can be viewed as \emph{divisible} discrete quantum subgroups of $\widehat{\mathbb{G}\times \mathbb{H}}$ and torsion-freeness is preserved under divisible discrete quantum subgroups as shown in \cite{RubenTorsionDivisibles}. However, it is an open problem to classify all torsion actions of $\mathbb{G}\times\mathbb{H}$. For example, if $\mathbb{G}=:G$ and $\mathbb{H}=:H$ are classical discrete groups, then it is well-known that the subgroups of their direct product $G\times H$ are \emph{not} always of the form $G'\times H'$ with $G'\leq G$ and $H'\leq H$ subgroups. It suffices to take $G=H$ and consider $G$ as a subgroup of $G\times G$ embedded diagonally. The general classification of the subgroups of a direct product of groups is given by the celebrated \emph{Goursat's lemma} (cf. \cite{Goursat}). In particular, not all finite subgroups of $G\times H$ are classified as direct product of finite subgroups of each factor. This already indicates that the classification of torsion actions of a quantum direct product $\mathbb{G}\times\mathbb{H}$ in terms of torsion actions of $\mathbb{G}$ and torsion actions of $\mathbb{H}$ would be a non trivial task.
	
	One could restrict attention to \emph{projective torsion}. A (non-trivial) projective torsion action of a compact quantum group $\mathbb{G}$ is given by a simple finite dimensional C$^*$-algebra equipped with an ergodic action of $\mathbb{G}$ (not equivariantly Morita equivalent to $\mathbb{C}$). It is shown in \cite{KennyNestRubenBCProjective} by the author in collaboration with K. De Commer and R. Nest that these actions by $\mathbb{G}$ are in bijective correspondence with (measurable) $2$-cocycles on $\mathbb{G}$, i.e. with twisted representations of $\mathbb{G}$. This is in complete analogy with the situation of classical compact groups. Accordingly, the classification of projective torsion of a quantum direct product $\mathbb{G}\times \mathbb{H}$ turns into the classification of (measurable) $2$-cocycles on $\mathbb{G}\times \mathbb{H}$. In the classical setting of discrete groups, such a classification goes back to work by K. Tahara \cite{Tahara}. The author has initiated an investigation in this direction in order to classify projective torsion of a quantum direct product, which will be the subject of a subsequent paper.
	
	On the other hand, in order to apply the Meyer-Nest strategy in the quantum setting, one needs a \emph{complementary pair} of localizing subcategories, $(\mathscr{L}_{\widehat{\mathbb{G}}},\mathscr{N}_{\widehat{\mathbb{G}}})$, where $\mathscr{L}_{\widehat{\mathbb{G}}}$ must encode the torsion phenomenon of $\widehat{\mathbb{G}}$. A candidate was proposed in \cite{MeyerNestTorsion} and \cite{VoigtBaumConnesAutomorphisms} for specific examples (see also \cite[Section 4.1.2]{RubenThesis} for a description for general discrete quantum groups), but it has been an open question whether the corresponding pair is complementary in $\mathscr{KK}^{\widehat{\mathbb{G}}}$, which prevented from having a definition of a quantum assembly map whenever $\widehat{\mathbb{G}}$ is not torsion-free. Recently, Y. Arano and A. Skalski \cite{YukiBCTorsion} have observed that the candidates for $\mathscr{L}_{\widehat{\mathbb{G}}}$ and $\mathscr{N}_{\widehat{\mathbb{G}}}$ form indeed a complementary pair of subcategories in $\mathscr{KK}^{\widehat{\mathbb{G}}}$, which allows to define a quantum assembly map for every discrete quantum group $\widehat{\mathbb{G}}$ (torsion-free or not). Moreover, following a different approach by studying the projective representation theory of a compact quantum group, the same conclusion is reached for permutation torsion-free discrete quantum groups by the author in collaboration with K. De Commer and R. Nest \cite{KennyNestRubenBCProjective}. More details about this formulation can be found in Section \ref{sec.QuantumBC}.
	\bigskip
	
	In the present paper we are able to improve several results in this matter appearing already in \cite{RubenSemiDirect}. For instance, the functor $\mathcal{Z}:\mathscr{KK}^{\widehat{\mathbb{G}}}\times \mathscr{KK}^{\widehat{\mathbb{H}}} \rightarrow \mathscr{KK}^{\widehat{\mathbb{G}\times \mathbb{H}}}$ given by the exterior tensor product of Kasparov triples allows to describe appropriately, \emph{and without any torsion-freeness assumption}, the quantum BC property for $\widehat{\mathbb{G}\times \mathbb{H}}$ in terms of the quantum BC property for $\widehat{\mathbb{G}}$ and $\widehat{\mathbb{H}}$. Namely, if $\widehat{\mathbb{G}}$ and $\widehat{\mathbb{H}}$ satisfy the strong quantum BC property, then we show that $\widehat{\mathbb{G}\times \mathbb{H}}$ satisfies the BC property with coefficients in $A\otimes B$, for all $\widehat{\mathbb{G}}$-C$^*$-algebra $A$ and all $\widehat{\mathbb{H}}$-C$^*$-algebra $B$ (cf. Theorem \ref{theo.StrongBCDirectProd}).  Accordingly, an analogous assertion about the \emph{usual} quantum BC property needs further hypothesis related to the Künneth formula in order to compute the K-theory of a tensor product. Therefore, we are lead to consider a (quantum) equivariant analogue of the Künneth class $\mathcal{N}$, say $\mathcal{N}_{\widehat{\mathbb{G}}}$, containing those $\widehat{\mathbb{G}}$-C$^*$-algebras which make possible such a K-theory computation. This is done in analogy to the work by J. Chabert, S. Echterhoff and H. Oyono-Oyono in \cite{ChabertEchterhoffOyono}. If $\widehat{\mathbb{G}}$ is a classical locally compact group $G$, then $\mathcal{N}_{\widehat{\mathbb{G}}}=\mathcal{N}_G$ as defined in \cite{ChabertEchterhoffOyono}. Furthermore, our approach, as based in the Meyer-Nest categorical framework, yields a characterisation of the objects in the equivariant Künneth class in terms of the non-equivariant one, up to replacing $A$ by a $\mathscr{L}_{G}$-simplicial approximation of $A$.  This study is contained in Section \ref{sec.KunnethFunctors}.
	
	In Section \ref{sec.BCKunneth} we generalise some key results appearing in \cite{ChabertEchterhoffOyono} about the connection between the BC property with the Künneth formula. Namely, let $A$ be a $\widehat{\mathbb{G}}$-C$^*$-algebra and $B$ a C$^*$-algebra. Then we show that $A\in \mathcal{N}_{\widehat{\mathbb{G}}}$ $\Leftrightarrow$ $A\underset{r}{\rtimes}\widehat{\mathbb{G}}\in\mathcal{N}$ provided that $\widehat{\mathbb{G}}$ satisfies the BC property with coefficients in $A\otimes B$ (cf. Proposition \ref{pro.BCKunneth}). One remarkable result in \cite{ChabertEchterhoffOyono} is the following permanence property of the BC property for a direct product of locally compact groups. Let $G$ and $H$ be two locally compact groups satisfying the BC property with trivial coefficients. If $C^*(G)$ or $C^*(H)$ belongs to $\mathcal{N}$, then $G\times H$ satisfies the BC property with trivial coefficients (cf. \cite[Theorem 5.3]{ChabertEchterhoffOyono}). The analogue statement for quantum groups is stated and proven in Theorem \ref{theo.BCDirectProducts}. In the quantum setting further hypotheses are needed concerning the behaviour of $\mathbb{C}$ with respect to the \emph{equivariant} Künneth formula. Namely, we have to assume that $\mathbb{C}\in\mathcal{N}_{\widehat{\mathbb{G}}}$. This supplementary condition is automatically fulfilled in the classical setting because $\mathbb{C}$ is a type I C$^*$-algebra and $\mathcal{N}_G$ contains all type I $G$-C$^*$-algebras by virtue of \cite[Theorem 0.1]{ChabertEchterhoffOyono}. In the quantum setting, a similar related result is Theorem \ref{theo.QuantumGroupCalgebraKunneth}. However, to the best knowledge of the author, it is not known for instance whether $\mathbb{C}\in\mathcal{N}_{\widehat{\mathbb{G}}}$ for every discrete quantum group $\widehat{\mathbb{G}}$. One reason for this is that in our approach the objects in $\mathcal{N}_{\widehat{\mathbb{G}}}$ are characterised in terms of objects in $\mathcal{N}$ up to a \emph{$\mathscr{L}_{\widehat{\mathbb{G}}}$-simplicial approximation}, which entails to study the localisation functor $L$ in relation with crossed products and the \emph{equivariant} Künneth class (cf. Remark \ref{rem.AnalogoyThm01} for an extended discussion). One possibility to do so might be to adapt the \emph{Going-Down technique} from \cite{ChabertEchterhoffOyono} based on Theorem \ref{theo.QuantumGroupCalgebraKunneth}.
	
	Finally, we make the observation that \cite[Theorem 5.2]{YukiBCTorsion} and \cite[Corollary 5.5]{YukiBCTorsion} can be also obtained for the Künneth class instead of the bootstrap class (cf. Theorem \ref{theo.QuantumGroupCalgebraKunneth}). In particular, one obtains that $C(\mathbb{G})\in\mathcal{N}$ as soon as $\widehat{\mathbb{G}}$ satisfies the \emph{strong} quantum BC property. In the classical setting, one can argue as follows. If $G$ satisfies the BC property with coefficients (\emph{a fortriori} when $G$ satisfies the \emph{strong} BC property), then we can apply \cite[Proposition 4.9]{ChabertEchterhoffOyono} (cf. Proposition \ref{pro.BCKunneth} for the quantum counterpart). In particular, since we always have $\mathbb{C}\in\mathcal{N}_G$ as explained above, then \cite[Proposition 4.9]{ChabertEchterhoffOyono} implies that $C^*(G)=\mathbb{C}\underset{r}{\rtimes} G\in\mathcal{N}$. This observation allows to put the Künneth formula to work by computing K-theory groups of the C$^*$-algebras defining quantum direct products in relevant examples based, for instance, on works by Voigt and Vergnioux-Voigt (cf. \cite{VoigtBaumConnesFree}, \cite{VoigtBaumConnesUnitaryFree}, \cite{VoigtBaumConnesAutomorphisms}). See Section \ref{sec.KTheoryComp}.
	
	\bigskip
\textsc{Acknowledgements.} The author is grateful to the anonymous referee for their valuable comments.

\section{\textsc{Preliminaries}}
	\subsection{Notations and conventions}\label{sec.NotationsConventions}
	Let us fix the notations and the conventions that we use throughout the whole article.
	
	Whenever $\mathscr{C}$ denotes a category, we shall assume that $\mathscr{C}$ is essentially small, so morphisms $Hom_{\mathscr{C}}(\cdot, \cdot)$ form sets. Given a category $\mathscr{C}$, we denote by $\mathscr{C}^{op}$ its opposite category. 
	We say that $\mathscr{C}$ is \emph{countable additive} if it is additive and it admits \emph{countable direct sums}. If $F$ is an \emph{additive} functor on an additive category, it is, by definition, compatible with \emph{finite} direct sums. The categories considered in the present paper are assumed to be countable additive. Whenever we require an additive functor to be compatible with \emph{infinite (countable)} direct sums, it will be explicitly indicated. We denote by $\mathscr{A}b$ the abelian category of abelian groups and by $\mathscr{A}b^{\mathbb{Z}/2}$ the abelian category of $\mathbb{Z}/2$-graded abelian groups. We denote by $\text{C}^*\text{-Alg}$ (resp. $\widehat{\mathbb{G}}\text{-C}^*\text{-Alg}$, where $\mathbb{G}$ is a compact quantum group) the category of separable C$^*$-algebras (resp. $\widehat{\mathbb{G}}\text{-C}^*$-algebras) with $*$-homomorphisms (resp. $\widehat{\mathbb{G}}$-equivariant $*$-homomorphisms) as morphisms. We write $\mathscr{D}\subset \mathscr{C}$ whenever $\mathscr{D}$ is a (full) subcategory of $\mathscr{C}$. We use the symbol $0$ to denote either the trivial (abelian) group, the trivial C$^*$-algebra, the trivial category or the zero object of an additive category. The context will distinguish the specific situation.
	
	If $E$ is a $\mathbb{C}$-vector space and $\mathcal{S}$ is a subset of vectors of $E$, then we write $span\, \mathcal{S}$ for the corresponding $\mathbb{C}$-vector subspace generated by $\mathcal{S}$. If $(E,||\cdot||)$ is a normed $\mathbb{C}$-vector space and $\mathcal{S}\subset E$ we write $[span\, \mathcal{S}]:=\overline{span}\,\mathcal{S}$.
	
	Let $H$ be a Hilbert space. We denote by $\mathcal{B}(H)$ (resp.\ $\mathcal{K}(H)$) the space of all linear bounded (resp.\ compact) operators on $H$. If $A$ is a C$^*$-algebra and $H$ a  Hilbert $A$-module, we denote by $\mathcal{L}_{A}(H)$ (resp.\ $\mathcal{K}_A(H)$) the space of all (resp.\ compact) adjointable operators on $H$. Hilbert $A$-modules are considered to be \emph{right $A$-modules}, so that the corresponding inner products are considered to be conjugate-linear on the left and linear on the right.
	
	All our C$^*$-algebras (except for obvious exceptions such as multiplier C$^*$-algebras and von Neumann algebras) are supposed to be \emph{separable} and all our Hilbert modules are supposed to be \emph{countably generated}. If $A$ is a C$^*$-algebra and $\mathcal{S}$ is a subset of elements in $A$, we write $C^*\langle \mathcal{S}\rangle := C^*\langle \mathcal{S}\cup \mathcal{S}^*\rangle $ for the corresponding C$^*$-subalgebra of $A$ generated by $\mathcal{S}$, that is, the intersection of all C$^*$-subalgebras of $A$ containing $\mathcal{S}$. The symbol $\otimes$ stands for the minimal tensor product of C$^*$-algebras and the exterior tensor product of Hilbert modules depending on the context. The symbol $\underset{\max}{\otimes}$ stands for the maximal tensor product of C$^*$-algebras. In any of the previous cases, the \emph{elementary tensors} in the corresponding tensor product are denoted simply by $\otimes$ and the context will distinguish the specific situation. If $A$ and $B$ are two C$^*$-algebras, $\Sigma:A\otimes B\longrightarrow B\otimes A$ denotes the flip map. The symbol $\Sigma$ is used as well for the suspension functor in the framework of triangulated categories. The context will distinguish the specific situation. We use systematically the leg numbering notation, so if $H$ is a Hilbert space then $X_{12} = X\otimes 1 \in \mathcal{B}(H^{\otimes 3})$ for $X\in \mathcal{B}(H^{\otimes 2})$, etc.
	
	If $S, A$ are C$^*$-algebras, we denote by $M(A)=\mathcal{L}_A(A)$ the multiplier algebra of $A$ and we put $\widetilde{M}(A\otimes S):=\{x\in M(A\otimes S)\ |\ x(id_A\otimes S)\subset A\otimes S\mbox{ and } (id_A\otimes S)x\subset A\otimes S\}$, which contains $A\otimes M(S)$. If $H$ is a Hilbert $A$-module, we put $M(H):=\mathcal{L}_A(A, H)$, which contains canonically $H\cong \mathcal{K}_A(A, H)$. We put $\widetilde{M}(H\otimes S):=\{X\in M(H\otimes S)\ |\ X(id_A\otimes S)\subset H\otimes S\mbox{ and } (id_H\otimes S)X\subset H\otimes S\}$, which contains $H\otimes M(S)$.
	
	If $\mathbb{G}=(C(\mathbb{G}), \Delta)$ is a compact quantum group, the set of all unitary equivalence classes of irreducible unitary finite dimensional representations of $\mathbb{G}$ is denoted by $\text{Irr}(\mathbb{G})$. The trivial representation of $\mathbb{G}$ is denoted by $\epsilon$. If $x\in \text{Irr}(\mathbb{G})$ is such a class, we write $u^x\in\mathcal{B}(H_x)\otimes C(\mathbb{G})$ for a representative of $x$ and $H_x$ for the finite dimensional Hilbert space on which $u^x$ acts (we write $dim(x):= n_x$ for the dimension of $H_x$). The matrix coefficients of $u^x$ with respect to an orthonormal basis $\{\xi^x_1,\ldots, \xi^x_{n_x}\}$ of $H_x$ are defined by $u^x_{ij}:=(\omega_{ij}\otimes id)(u^x)$ for all $i,j=1,\ldots, n_x$. The linear span of matrix coefficients of all finite dimensional representations of $\mathbb{G}$ is denoted by $\text{Pol}(\mathbb{G})$, which is a Hopf $*$-algebra with co-multipliction $\Delta$ and co-unit and antipode denoted by $\varepsilon_\mathbb{G}$ and $S_\mathbb{G}$, respectively. Given $x,y\in \text{Irr}(\mathbb{G})$, the tensor product of $x$ and $y$ is denoted by $x\otimes y$.
	
	The Haar state of $\mathbb{G}$ is denoted by $h_{\mathbb{G}}$. The GNS construction corresponding to $h_{\mathbb{G}}$ is denoted by $(L^2(\mathbb{G}), \lambda, \xi_{\mathbb{G}})$. We also write $\Lambda(x) = \lambda(x)\xi_{\mathbb{G}}$ for $x\in C(\mathbb{G})$. We adopt the standard convention for the inner product on $L^2(\mathbb{G})$, which means that $\langle \Lambda(x), \Lambda(y)\rangle:=h_{\mathbb{G}}(x^*y)$ for all $x,y\in C(\mathbb{G})$. We suppress the notation $\lambda$ in computations so that we simply write $x\Lambda(y)=\Lambda(xy)$ for all $x,y\in C(\mathbb{G})$. We will make the standing assumption that $h_{\mathbb{G}}$ is faithful, so we only work with the reduced form of a compact quantum group, hence $C_r(\mathbb{G})=C(\mathbb{G})$ unless the contrary is specified. The maximal form of $\mathbb{G}$ is given by the C$^*$-envelopping algebra of $\text{Pol}(\mathbb{G})$, denoted by $C_m(\mathbb{G})$.
	
	The Haar state extends uniquely to a normal faithful state on $L^{\infty}(\mathbb{G})$, and we denote by $J_{\mathbb{G}}$ the associated modular conjugation on $L^2(\mathbb{G})$. Let $I_0$ be the anti-linear involutive map $\Lambda(\text{Pol}(\mathbb{G})) \rightarrow L^2(\mathbb{G})$ defined by $\Lambda(x) \mapsto \Lambda(S(x)^*)$ for $x\in \text{Pol}(\mathbb{G})$. Then $I_0$ is closeable, and we denote $I = \widehat{J}_{\mathbb{G}}|I|$ for the polar decomposition of its closure. The map $R(x)=\widehat{J}_{\mathbb{G}}x^* \widehat{J}_{\mathbb{G}}$, for all $x\in C(\mathbb{G})$, is a well-defined anti-multiplicative and anti-co-multiplicative map on $C(\mathbb{G})$ preserving $\text{Pol}(\mathbb{G})$, called \emph{unitary antipode}. We put $U_{\mathbb{G}} = J_{\mathbb{G}}\widehat{J}_{\mathbb{G}} = \widehat{J}_{\mathbb{G}}J_{\mathbb{G}}\in \mathcal{B}(L^2(\mathbb{G}))$ for the symmetry of the standard Kac system associated to $\mathbb{G}$. The we put $\rho(a) := U_{\mathbb{G}}\lambda(a)U_{\mathbb{G}}$, for all $a\in C(\mathbb{G})$.
	
	Given a compact quantum group $\mathbb{G}$, we put $c_0(\widehat{\mathbb{G}}):= [\{(id\otimes \eta)(V_{\mathbb{G}})\ |\ \eta\in \mathcal{B}(L^2(\mathbb{G}))_*\}]\subset \mathcal{B}(L^2(\mathbb{G}))$, where $V_\mathbb{G}$ denotes the right regular representation of $\mathbb{G}$ on $L^2(\mathbb{G})$. Recall that $c_0(\widehat{\mathbb{G}})$ is a C$^*$-algebra which defines a locally compact quantum group with co-multiplication $\widehat{\Delta}(x) :=\Sigma V_{\mathbb{G}}^*(1\otimes x)V_{\mathbb{G}}\Sigma $, for all $x\in c_0(\widehat{\mathbb{G}})$. There exists a natural isomorphism $c_0(\widehat{\mathbb{G}})\cong \underset{x\in Irr(\mathbb{G})}{\bigoplus^{c_0}} \mathcal{B}(H_x)$. We denote the identity map by $\widehat{\lambda}: c_0(\widehat{\mathbb{G}}) \rightarrow \mathcal{B}(L^2(\mathbb{G}))$.
	
	If $\mathbb{H}$ is another compact quantum group, we say that $\widehat{\mathbb{H}}$ is a discrete quantum subgroup of $\widehat{\mathbb{G}}$ if one (hence all) of the following conditions hold: $i)$ $Pol(\mathbb{H})$ is a Hopf $*$-subalgebra of $\text{Pol}(\mathbb{G})$, $ii)$ $C_r(\mathbb{H})\overset{\iota}{\subset} C_r(\mathbb{G})$ such that $\iota$ intertwines the co-multiplications, $iii)$ $C_m(\mathbb{H})\overset{\iota}{\subset} C_m(\mathbb{G})$ such that $\iota$ intertwines the co-multiplications; $iv)$ $\mathscr{R}\text{ep}(\mathbb{H})$ is a full subcategory of $\mathscr{R}\text{ep}(\mathbb{G})$ containing the trivial representation and stable by direct sums, tensor product and adjoint operations. See \cite{SoltanSubgroups} for more details. In this case we write $\widehat{\mathbb{H}}<\widehat{\mathbb{G}}$. Note that in this case we have $\epsilon:=\epsilon_\mathbb{G}=\epsilon_{\mathbb{H}}$. The trivial quantum subgroup of $\widehat{\mathbb{G}}$ is denoted by $\mathbb{E}$.

	\subsection{Actions of quantum groups}\label{sec.ActionsQG}
	In this section we recall some terminology about actions of quantum groups. We refer to \cite{KennyActions} for more details. 
	\begin{defi}
			Let $\mathbb{G}=(C(\mathbb{G}),\Delta)$ be a compact quantum group. A right $\mathbb{G}$-C$^*$-algebra is a C$^*$-algebra $A$ together with an injective non-degenerate $*$-homomorphism $\alpha: A\longrightarrow A\otimes C(\mathbb{G})$ such that: $i)$ $(\alpha\otimes id_{C(\mathbb{G})})\circ\alpha = (id_A\otimes \Delta)\circ\alpha$ and $ii)$ $[\alpha(A)(1\otimes C(\mathbb{G}))]=A\otimes C(\mathbb{G})$. Such homomorphism is called a \emph{right action of $\mathbb{G}$ on $A$} or a \emph{right co-action of $C(\mathbb{G})$ on $A$}.	
	\end{defi}
	
	Similarly, we can define a \emph{left} action of $\mathbb{G}$ on $A$ (or a \emph{left} co-action of $C(\mathbb{G})$ on $A$) as a non-degenerate $*$-homomorphism $\alpha: A\longrightarrow C(\mathbb{G})\otimes A$ satisfying the analogous properties of the preceding definition. In the present article, an action of a compact quantum group $\mathbb{G}$ is supposed to be a \emph{right} one unless the contrary is explicitly indicated. Hence, we refer to such actions simply as \emph{actions of $\mathbb{G}$}. Observe however that if $(A, \alpha)$ is a \emph{right} $\mathbb{G}$-C$^*$-algebra, then $(A^{op}, \overline{\alpha})$ is a \emph{left} $\mathbb{G}$-C$^*$-algebra where $A^{op}$ denotes the opposite C$^*$-algebra of $A$ and 
$\overline{\alpha}: A^{op}\longrightarrow C(\mathbb{G})\otimes A^{op}$ is defined by $\overline{\alpha}:=(R\otimes id)\circ\Sigma\circ\alpha$, where $R$ denotes the unitary antipode of $\mathbb{G}$.

	We also recall the notion of action for discrete quantum groups for the sake of completeness. 
	\begin{defi}
		Let $\mathbb{G}=(C(\mathbb{G}),\Delta)$ be a compact quantum group. A right $\widehat{\mathbb{G}}$-C$^*$-algebra is a C$^*$-algebra $A$ together with an injective non-degenerate $*$-homomorphism $\alpha: A\longrightarrow \widetilde{M}(A\otimes c_0(\widehat{\mathbb{G}}))$ such that: $i)$ $(\alpha\otimes id_{c_0(\widehat{\mathbb{G}})})\circ\alpha = (id_A\otimes \widehat{\Delta})\circ\alpha$ and $ii)$ $[\alpha(A)(1\otimes c_0(\widehat{\mathbb{G}}))]=A\otimes c_0(\widehat{\mathbb{G}})$. Such homomorphism is called a \emph{right action of $\widehat{\mathbb{G}}$ on $A$} or a \emph{right co-action of $c_0(\widehat{\mathbb{G}})$ on $A$}.
		\end{defi}
	Again, one has the analogous notion of a left action of $\widehat{\mathbb{G}}$.  In the following, an action of a discrete quantum group $\widehat{\mathbb{G}}$ is supposed to be a \emph{right} one unless the contrary is explicitly indicated.
	\bigskip

	We also recall the notion of equivariant Hilbert module with respect to a compact quantum group for the sake of completeness.
	\begin{defi}\label{defi.EquivariantModule}
		Let $\mathbb{G}$ be a compact quantum group and $(A, \delta)$ a $\mathbb{G}$-C$^*$-algebra. A right $\mathbb{G}$-equivariant Hilbert $A$-module is a right $A$-module $E$ together with an injective linear map $\delta_E: E\longrightarrow E\otimes C(\mathbb{G})$ such that: $i)$ $\delta_E(\xi\cdot a)=\delta_E(\xi)\circ\delta(a)$ for all $\xi\in E$ and $a\in A$; $ii)$ $\delta\big(\langle \xi, \eta\rangle\big)=\langle \delta_E(\xi), \delta_E(\eta) \rangle$ for all $\xi, \eta\in E$; $iii)$ $(\delta_E\otimes id)\circ \delta_E=(id\otimes \Delta)\circ \delta_E$; $iv)$ $[\delta_E(E)(A\otimes C(\mathbb{G}))]=E\otimes C(\mathbb{G})$. Such map is called a \emph{right action of $\mathbb{G}$ on $E$} or a \emph{right co-action of $C(\mathbb{G})$ on $E$}.
	\end{defi}
	\begin{rem}\label{rem.RepGEquivHilbSpace}
		The notion of representation of a compact quantum group can be viewed alternatively as follows. If $u\in\mathcal{B}(H)\otimes C(\mathbb{G})$ is a (finite dimensional) unitary representation of $\mathbb{G}$ on a (finite dimensional) Hilbert space $H$, then the map $\delta_H: H\rightarrow H\otimes C(\mathbb{H})$, $\xi\mapsto \delta_H(\xi):=u(\xi\otimes id)$, turns $H$ into a $\mathbb{G}$-equivariant Hilbert space. Conversely, if $(H, \delta_H)$ is a (finite dimensional) $\mathbb{G}$-equivariant Hilbert space, then the unitary $u\in\mathcal{B}(H)\otimes C(\mathbb{G})$ defined by $u(\xi\otimes 1)=\delta_H(\xi)$, for all $\xi\in H$ is a (finite dimensional) unitary representation of $\mathbb{G}$ on $H$.
	\end{rem}
	\begin{defi}\label{defi.EquivariantModule}
		Let $\mathbb{G}$ be a compact quantum group and $(A, \delta)$ a $\mathbb{G}$-C$^*$-algebra. Let $(E, \delta_E)$ be a $\mathbb{G}$-equivariant Hilbert $A$-module. We say that $E$ is irreducible if the space of equivariant adjointable operators of $E$, $\mathcal{L}_{\mathbb{G}}(E):=\{T\in\mathcal{L}_A(E)\ |\ \delta_E(T(\xi))=(T\otimes 1)\delta_E(\xi)\mbox{, for all $\xi\in E$}\},$ is one-dimensional.
	\end{defi}
	\begin{rem}\label{rem.AdmissibleUnitary}
		If $(E, \delta_E)$ is a $\mathbb{G}$-equivariant Hilbert $A$-module as above, then $\mathcal{K}_A(E)$ is a $\mathbb{G}$-C$^*$-algebra with action $\delta_{\mathcal{K}_A(E)}$ defined by $\delta_{\mathcal{K}_A(E)}(\theta_{\xi, \eta})=\delta_E(\xi)\delta_E(\eta)^*\in \mathcal{K}_A(E)\otimes C(\mathbb{G})$, for all $\xi, \eta\in E$ where $\theta_{\xi, \eta}$ denotes the corresponding rank one operator in $E$. By abuse of notation, we still denote by $\delta_{\mathcal{K}_A(E)}$ the extension of this homomorphism to $\mathcal{L}_A(E) = M(\mathcal{K}_A(E)) \rightarrow M(\mathcal{K}_{A}(E) \otimes C(\mathbb{G}))$. The latter is however not in general an action of $\mathbb{G}$ on $\mathcal{L}_A(E)$. Recall further that giving an action $\delta_E$ is equivalent to giving a unitary operator $V_E\in\mathcal{L}_{A\otimes C(\mathbb{G})}\big(E\underset{\delta}{\otimes}(A\otimes C(\mathbb{G})), E\otimes C(\mathbb{G})\big)$ such that $\delta_E(\xi)=V_E\circ T_\xi$ for all $\xi\in E$ where $T_\xi\in\mathcal{L}_{A\otimes C(\mathbb{G})}(A\otimes C(\mathbb{G}), E\underset{\delta}{\otimes}(A\otimes C(\mathbb{G})))$ is such that $T_\xi(x)=\xi\underset{\delta}{\otimes} x$, for all $x\in A\otimes C(\mathbb{G})$. One calls $V_E$ the \emph{admissible operator for $(E, \delta_E)$}. Moreover, we have $\delta_{\mathcal{K}_A(E)}=Ad_{V_E}$. We refer to \cite{BaajSkandalisQuantumKK} for more details. Note that $\mathcal{L}_{\mathbb{G}}(E)=\mathcal{L}_A(E)^{Ad_{V_E}}$. So, if $(E, \delta_E)$ is irreducible, then $\mathcal{L}_A(E)=\mathcal{K}_A(E)$ together with $Ad_{V_E}$ defines an ergodic action of $\mathbb{G}$.
	\end{rem}
	
	\begin{defi}
		Let $\mathbb{G}$ be a compact quantum group. Let $(A, \alpha)$ and $(B, \beta)$ be two $\mathbb{G}$-C$^*$-algebras. We say that $A$ and $B$ are $\mathbb{G}$-equivariantly Morita equivalent if there exists a $\mathbb{G}$-equivariant Hilbert $A$-module $(E, \delta_E)$ such that $B\cong \mathcal{K}_A(E)$ as $\mathbb{G}$-C$^*$-algebras. In this case we write $A\underset{\mathbb{G}}{\sim} B$.
	\end{defi}

	If $(A, \alpha)$ is a $\mathbb{G}$-C$^*$-algebra, we put $A^\alpha:=\{a\in A\ |\ \alpha(a)=a\otimes 1_{C(\mathbb{G})}\}$, which is a sub-C$^*$-algebra of $A$. If $A$ is unital, we say that $\alpha$ is ergodic if $A^\alpha=\mathbb{C}\cdot 1_A$. 
	
	\begin{defi}
		Let $\mathbb{G}$ be a compact quantum group. A torsion action of $\mathbb{G}$ is a $\mathbb{G}$-C$^*$-algebra $(T, \tau)$ where $T$ is finite dimensional and $\tau$ is ergodic. We say that $\widehat{\mathbb{G}}$ is torsion-free if any torsion action of $\mathbb{G}$ is $\mathbb{G}$-equivariantly Morita equivalent to the trivial $\mathbb{G}$-C$^*$-algebra $\mathbb{C}$. The set of all equivalence classes under equivariant Morita equivalence of non-trivial torsion actions of $\mathbb{G}$ is denoted by $\text{Tor}(\widehat{\mathbb{G}})$.
	\end{defi}
	\begin{note}
		Notice that there is no risk of confusion between the notation $\text{Tor}(\widehat{\mathbb{G}})$ for torsion actions of $\mathbb{G}$ and the usual $\text{Tor}$-functor.
	\end{note}
	
	Finally, let us recall the \emph{two-sided crossed product} construction. It appears already in \cite[Section 2.6]{NikshychVainerman} in the context of quantum groupoids. It is used to formulate a quantum Baum-Connes assembly map in \cite{YukiBCTorsion} and \cite{KennyNestRubenBCProjective}, which we will use later. Let $\mathbb{G}$ be a compact quantum group. If $(B, \beta)$ is a right $\mathbb{G}$-C$^*$-algebra and $(A, \alpha)$ is a left $\mathbb{G}$-C$^*$-algebra, then the two-sided crossed product of $B$ and $A$ by $\mathbb{G}$, denoted by $B\underset{r, \beta}{\rtimes}\mathbb{G}\underset{r, \alpha}{\ltimes}A$, is the C$^*$-algebra defined by:
		$$B\underset{r, \beta}{\rtimes}\mathbb{G}\underset{r, \alpha}{\ltimes}A:=C^*\langle ((id\otimes \lambda)\beta(B)\otimes 1)(1\otimes \widehat{\lambda}(c_0(\widehat{\mathbb{G}}))\otimes 1)(1\otimes (\rho\otimes id)(\alpha(A)))\rangle\subset \mathcal{L}_{B\otimes A}(B\otimes L^2(\mathbb{G})\otimes A).$$
	
	To lighten the notations we will omit the representations $\lambda$, $\widehat{\lambda}$ and $\rho$ in the definition of $B\underset{r, \beta}{\rtimes}\mathbb{G}\underset{r, \alpha}{\ltimes}A$, and note that then $\rho(x) = U_{\mathbb{G}}xU_{\mathbb{G}}$ for $x\in C(\mathbb{G})$. We also write $\alpha_U(x) = (U_{\mathbb{G}}\otimes id)\alpha(x)(U_{\mathbb{G}}\otimes id)$ for $x\in A$. It is easy to show that $B\underset{r, \beta}{\rtimes}\mathbb{G}\underset{r, \alpha}{\ltimes}A=\overline{span}\{(\beta(B)\otimes 1)(1\otimes c_0(\widehat{\mathbb{G}})\otimes 1)(1\otimes \alpha_U(A))\}$. We use these two descriptions of $B\underset{r, \beta}{\rtimes}\mathbb{G}\underset{r, \alpha}{\ltimes}A$ interchangeably.
	
	\subsection{Quantum direct products}\label{sec.QDP}

	Let us recall briefly  the construction of quantum direct product in the sense of S. Wang \cite{WangSemidirect}. 
	
	Let $\mathbb{G}$ and $\mathbb{H}$ be two compact quantum groups. The quantum direct product of $\mathbb{G}$ and $\mathbb{H}$ is a compact quantum group denoted by $\mathbb{F}:=\mathbb{G}\times\mathbb{H}$ with $C(\mathbb{F})=C(\mathbb{G})\otimes C(\mathbb{H})$. The co-multiplication on $\mathbb{F}$ is given by $\Delta_{\mathbb{F}}=(id\otimes \Sigma\otimes id)(\Delta_{\mathbb{G}}\otimes \Delta_{\mathbb{H}})$. The representation theory of $\mathbb{F}$ can be described explicitly as follows. For every irreducible representation $z\in \text{Irr}(\mathbb{F})$, take a representative $u^z\in \mathcal{B}(H_z)\otimes C(\mathbb{F})$. Then there exist unique irreducible representations $x\in \text{Irr}(\mathbb{G})$ and $y\in \text{Irr}(\mathbb{H})$ such that if $u^x\in \mathcal{B}(H_x)\otimes C(\mathbb{G})$ and $u^y\in \mathcal{B}(H_y)\otimes C(\mathbb{H})$ are respective representatives of $x$ and $y$, then we have $u^z\cong u^x_{13}u^y_{24}\in\mathcal{B}(H_x\otimes H_y)\otimes C(\mathbb{F})\mbox{,}$ where $u^x_{13}$ and $u^y_{24}$ are the corresponding legs of $u^x$ and $u^y$, respectively inside $\mathcal{B}(H_x)\otimes\mathcal{B}(H_y)\otimes C(\mathbb{G})\otimes C(\mathbb{H})$. In this case we write $u^{(x,y)}:=u^{z}$.
	
	
	This description of $\text{Irr}(\mathbb{F})$ yields that $c_0(\widehat{\mathbb{F}})\cong c_0(\widehat{\mathbb{G}})\otimes c_0(\widehat{\mathbb{H}})$.Therefore, given a $\widehat{\mathbb{G}}$-C$^*$-algebra $(A, \alpha)$ and a $\widehat{\mathbb{H}}$-C$^*$-algebra $(B, \beta)$, the tensor product $A\otimes B$ is a C$^*$-algebra equipped with the following action of $\widehat{\mathbb{F}}$:
	\begin{equation*}
			\begin{split}
				A\otimes B\overset{\alpha\otimes\beta}{\longrightarrow}&\widetilde{M}(c_0(\widehat{\mathbb{G}})\otimes A)\otimes \widetilde{M}(c_0(\widehat{\mathbb{H}})\otimes B)\subset \widetilde{M}(c_0(\widehat{\mathbb{G}})\otimes A\otimes c_0(\widehat{\mathbb{H}})\otimes B)\\
				&\overset{\Sigma_{23}}{\cong} \widetilde{M}(c_0(\widehat{\mathbb{G}})\otimes c_0(\widehat{\mathbb{H}})\otimes A\otimes B)\cong \widetilde{M}(c_0(\widehat{\mathbb{F}})\otimes A\otimes B).	
			\end{split}
		\end{equation*}
		
		By abuse of notation, we denote this composition simply by $\alpha\otimes\beta$. In a similar way, given a $\mathbb{G}$-C$^*$-algebra $(T, \tau)$ and a $\mathbb{H}$-C$^*$-algebra $(S, \sigma)$ we can still define the $\mathbb{F}$-C$^*$-algebra $(T\otimes S, \tau\otimes \sigma)$. In particular, it is clear that if $(T,\tau)$ is a torsion action of $\mathbb{G}$ and $(S, \sigma)$ is a torsion action of $\mathbb{H}$, then $(R, \rho)$ is a torsion action of $\mathbb{F}$.

	Finally, observe that the quantum direct product construction can be done more generally for locally compact quantum groups too. In this respect, the following is an easy observation (see \cite[Corollary 2.3.3]{RubenSemiDirect} for a proof in the case of discrete quantum groups).
	\begin{proSec}\label{pro.DirectProductTensorProduct}
		Let $\mathbb{G}$, $\mathbb{H}$ be two locally compact quantum groups and let $\mathbb{F}:=\mathbb{G}\times \mathbb{H}$ be the corresponding quantum direct product of $\mathbb{G}$ and $\mathbb{H}$. If $(A,\alpha)$ is a $\mathbb{G}$-C$^*$-algebra and $(B,\beta)$ is a $\mathbb{H}$-C$^*$-algebra, then there exists a canonical $*$-isomorphism $C \underset{\delta,r}{\rtimes}\mathbb{F} \cong A\underset{\alpha,r}{\rtimes}\mathbb{G} \otimes B\underset{\beta,r}{\rtimes} \mathbb{H}\mbox{,}$ where $C:=A\otimes B$ is the $\mathbb{F}$-C$^*$-algebra with action $\delta:=\alpha\otimes \beta$.
		
		In particular, if $\mathbb{H}=\mathbb{E}$ is the trivial locally compact quantum group and $B$ is any C$^*$-algebra equipped with trivial action, then there exists a canonical $*$-isomorphism $(A\otimes B) \underset{\delta,r}{\rtimes}\mathbb{G} \cong A\underset{\alpha,r}{\rtimes}\mathbb{G} \otimes B$, where $A\otimes B$ is the $\mathbb{G}$-C$^*$-algebra with action $\delta:=\alpha\otimes id$.
	\end{proSec}
	\begin{proof}
		The isomorphism of the statement is simply induced by the canonical map $\mathcal{L}_A(L^2(\mathbb{G})\otimes A)\otimes \mathcal{L}_B(L^2(\mathbb{H})\otimes B)\rightarrow \mathcal{L}_{A\otimes B}\big(L^2(\mathbb{G})\otimes L^2(\mathbb{H})\otimes A\otimes B\big)=\mathcal{L}_{A\otimes B}(L^2(\mathbb{F})\otimes C)$.
	\end{proof}
	
	In relation with the two-sided crossed product construction recalled in Section \ref{sec.NotationsConventions}, the following result (which is a generalisation of Proposition \ref{pro.DirectProductTensorProduct}) is straightforward after a routine computation by applying the definitions and the fact that $c_0(\widehat{\mathbb{F}})\cong c_0(\widehat{\mathbb{G}})\otimes c_0(\widehat{\mathbb{H}})$.
	\begin{proSec}\label{pro.TwoSidedDirectProd}
	Let $\mathbb{G}$ and $\mathbb{H}$ be two compact quantum groups and let $\mathbb{F}:=\mathbb{G}\times\mathbb{H}$ be the corresponding quantum direct product of $\mathbb{G}$ and $\mathbb{H}$. Let $(A,\alpha)$ and $(B,\beta)$ be a right $\mathbb{G}$-C$^*$-algebra and a right $\mathbb{H}$-C$^*$-algebra, respectively. Let $(T, \tau)$ and $(S, \sigma)$ be a left $\mathbb{G}$-C$^*$-algebra and a left $\mathbb{H}$-C$^*$-algebra. Then $(A\otimes B)\underset{r, \alpha\otimes\beta}{\rtimes}\mathbb{F}\underset{r, \tau\otimes \sigma}{\ltimes}(T\otimes S)\cong (A\underset{r, \alpha}{\rtimes}\mathbb{G}\underset{r, \tau}{\ltimes}T)\otimes (B\underset{r, \beta}{\rtimes}\mathbb{H}\underset{r, \sigma}{\ltimes}S)$.
	
	In particular, if $\mathbb{H}=\mathbb{E}$ is the trivial compact quantum group and $B$ is any C$^*$-algebra equipped with trivial action, then $(A\otimes B)\underset{r, \alpha\otimes id}{\rtimes}\mathbb{G}\underset{r, \tau}{\ltimes}T\cong (A\underset{r, \alpha}{\rtimes}\mathbb{G}\underset{r, \tau}{\ltimes}T)\otimes B$.
	\end{proSec}
	
	\begin{note}
		In the original paper \cite{WangSemidirect} by S. Wang the construction $\mathbb{G}\times\mathbb{H}$ is called \emph{Cartesian product of $\mathbb{G}$ by $\mathbb{H}$}. Moreover, this construction is consistent with the construction of direct product of classical discrete groups, for instance. Hence, the terminology \emph{direct product of $\mathbb{G}$ by $\mathbb{H}$} should not cause any confusion. However, we use the terminology \emph{quantum direct product} to emphasise the use of arbitrary compact quantum groups in this construction.
	\end{note}

	\subsection{The equivariant Kasparov category and the quantum Baum-Connes assembly map}\label{sec.QuantumBC}
	We refer to \cite{Blackadar}, \cite{BaajSkandalisQuantumKK} for more details about (equivariant) $KK$-theory and to  \cite{Jorgensen}, \cite{MeyerNest} for more details about triangulated categories and the Meyer-Nest approach to the BC property, which is fundamental for a quantum counterpart of the conjecture. 
	
	Let $\widehat{\mathbb{G}}$ be a discrete quantum group and consider the corresponding equivariant Kasparov category, $\mathscr{KK}^{\widehat{\mathbb{G}}}$, with canonical suspension functor denoted by $\Sigma$. $\mathscr{KK}^{\widehat{\mathbb{G}}}$ is a triangulated category whose distinguished triangles are given by mapping cone triangles (see \cite{MeyerNest} for more details). The word \emph{homomorphism (resp.\ isomorphism)} will mean \emph{homomorphism (resp.\ isomorphism) in the corresponding Kasparov category}; it will be a true homomorphism (resp.\ isomorphism) between $\widehat{\mathbb{G}}$-C$^*$-algebras or any Kasparov triple between $\widehat{\mathbb{G}}$-C$^*$-algebras (resp.\ any equivariant $KK$-equivalence between $\widehat{\mathbb{G}}$-C$^*$-algebras). Analogously, we can consider the equivariant Kasparov category $\mathscr{KK}^{\mathbb{G}}$.
	
	Exterior tensor product of Kasparov triples is important for the purpose of the present paper. Namely, given C$^*$-algebras $A$, $A'$, $B$, $B'$ and $D$, the map:
						$$
						\begin{array}{rccl}
							\widetilde{\tau}_D:&KK(A,A'\otimes D)\times KK(D\otimes B, B') & \longrightarrow &KK(A\otimes B, A'\otimes B')\\
							&(\mathcal{X}, \mathcal{Y}) & \longmapsto &\widetilde{\tau}_D(\mathcal{X}, \mathcal{Y}):=\tau_{B}(\mathcal{X})\underset{A'\otimes D\otimes B}{\otimes}{_{A'}}\tau(\mathcal{Y})=:\mathcal{X}\underset{D}{\otimes^{\text{ext}}}\mathcal{Y}
						\end{array}
						$$
						is bilinear, contravariantly functorial in $A$ and $B$ and covariantly functorial in $A'$ and $B'$. Here $\underset{A'\otimes D\otimes B}{\otimes}$ denotes the usual Kasparov product over $A'\otimes D\otimes B$; and $\tau_B$ and ${_{A'}}\tau$ denote the right and left exterior tensor products by $B$ and $A'$, respectively (cf. \cite[Chapter VIII, Section 18]{Blackadar} for details). Notice however that $\tau_A\cong{_{A}}\tau$, for all C$^*$-algebra $A$ because the tensor product of C$^*$-algebras is a commutative operation. In particular, taking $D:=\mathbb{C}$, $\widetilde{\tau}_{\mathbb{C}}$ defines a \emph{tensor product of Kasparov triples}. Namely, by virtue of \cite[Proposition 18.9.1]{Blackadar}, the operation $\widetilde{\tau}_{\mathbb{C}}(\ \cdot\ )= (\ \cdot\ ) \underset{\mathbb{C}}{\otimes^{\text{ext}}} (\ \cdot\ )$ is associative. For the sake of completeness, let us check the latter claim by translating the formulas from \cite[Proposition 18.9.1]{Blackadar} to the case $\widetilde{\tau}_{\mathbb{C}}$. Let $A$, $A'$, $B$, $B'$, $C$ and $C'$ be C$^*$-algebras and $\mathcal{X}\in KK(A, A') $, $\mathcal{Y}\in KK(B, B')$, $\mathcal{Z}\in KK(C, C')$. Then:
	\begin{equation*}
	\begin{split}
		(\mathcal{X}\underset{\mathbb{C}}{\otimes^{\text{ext}}}\mathcal{Y})\underset{\mathbb{C}}{\otimes^{\text{ext}}}\mathcal{Z}&=\tau_{C}\Big(\mathcal{X}\underset{\mathbb{C}}{\otimes^{\text{ext}}}\mathcal{Y}\Big)\underset{A'\otimes B'\otimes C}{\otimes} {_{A'\otimes B'}}\tau(\mathcal{Z})\\
		&=\Big(\tau_C(\mathcal{X})\underset{C}{\otimes^{\text{ext}}}{_{C}}\tau(\mathcal{Y})\Big)\underset{A'\otimes B'\otimes C}{\otimes} {_{A'\otimes B'}}\tau(\mathcal{Z})\\
		&=\Big(\tau_B\big(\tau_C(\mathcal{X})\big)\underset{A'\otimes B\otimes C}{\otimes} {_{A'}}\tau\big({_C}\tau(\mathcal{Y})\big)\Big)\underset{A'\otimes B'\otimes C}{\otimes} {_{A'\otimes B'}}\tau(\mathcal{Z})\\
		&=\Big(\tau_{B\otimes C}(\mathcal{X})\underset{A'\otimes B\otimes C}{\otimes} {_{C\otimes A'}}\tau(\mathcal{Y})\Big)\underset{A'\otimes B'\otimes C}{\otimes} {_{A'\otimes B'}}\tau(\mathcal{Z})\\
		&=\tau_{B\otimes C}(\mathcal{X})\underset{A'\otimes B\otimes C}{\otimes} \Big({_{C\otimes A'}}\tau(\mathcal{Y})\underset{A'\otimes B'\otimes C}{\otimes} {_{A'\otimes B'}}\tau(\mathcal{Z})\Big)\\
		&=\tau_{B\otimes C}(\mathcal{X})\underset{A'\otimes B\otimes C}{\otimes} \Big(\tau_C\big(\tau_{A'}(\mathcal{Y})\big)\underset{A'\otimes B'\otimes C}{\otimes} {_{B'}}\tau\big({_{A'}}\tau(\mathcal{Z})\big)\Big)\\
		&=\tau_{B\otimes C}(\mathcal{X})\underset{A'\otimes B\otimes C}{\otimes} \Big(\tau_{A'}(\mathcal{Y})\underset{A'}{\otimes^{\text{ext}}}{_{A'}}\tau(\mathcal{Z})\Big)\\
		&=\tau_{B\otimes C}(\mathcal{X})\underset{A'\otimes B\otimes C}{\otimes} {_{A'}}\tau\Big(\mathcal{Y}\underset{\mathbb{C}}{\otimes^{\text{ext}}}\mathcal{Z}\Big)=\mathcal{X}\underset{\mathbb{C}}{\otimes^{\text{ext}}}(\mathcal{Y}\underset{\mathbb{C}}{\otimes^{\text{ext}}}\mathcal{Z}).
	\end{split}
	\end{equation*}
	
	Moreover, given a C$^*$-algebra $A$, the functor $A\otimes (\ \cdot\ ):\text{C}^*\text{-Alg}\rightarrow \mathscr{KK}$, $B\mapsto A\otimes B$, is a split-exact, stable and homotopy invariant. By the universal property of the Kasparov category (cf. \cite{HigsonUnivProp}), it extends to a functor $A\otimes (\ \cdot\ ): \mathscr{KK}\rightarrow \mathscr{KK}$. Similarly, we have a functor $(\ \cdot\ )\otimes A: \mathscr{KK}\rightarrow \mathscr{KK}$. Since these extensions are natural, we obtain a bifunctor $\widetilde{\tau}_{\mathbb{C}}(\ \cdot\ )= (\ \cdot\ ) \underset{\mathbb{C}}{\otimes^{\text{ext}}} (\ \cdot\ ): \mathscr{KK}\times  \mathscr{KK}\rightarrow \mathscr{KK}$. In other words, the bifunctor $\widetilde{\tau}_{\mathbb{C}}$ confer a monoidal structure on $\mathscr{KK}$. Finally, since the Kasparov product plays the role of composition of morphisms within the category $\mathscr{KK}$, functoriality of $\widetilde{\tau}$ yields that:
	\begin{equation}\label{eq.TensorProdKasparovProd}
	\begin{split}
		(y \underset{\mathbb{C}}{\otimes^{\text{ext}}} x)\underset{A\otimes B}{\otimes} (z  \underset{\mathbb{C}}{\otimes^{\text{ext}}} w)=(y\underset{A}{\otimes} z)  \underset{\mathbb{C}}{\otimes^{\text{ext}}} (x\underset{B}{\otimes} w),
	\end{split}
	\end{equation}
	for all $y\in KK(D, A)$, $x\in KK(D', B)$, $z\in KK(A, A')$, $w\in KK(B, B')$ and all C$^*$-algebras $A$, $A'$, $B$, $B'$, $D$, $D'$.
	\begin{rem}
		If $G$ is a locally compact group, then the equivariant Kasparov category $\mathscr{KK}^{G}$ is also characterised by a universal property in terms of split-exactness, stability and homotopy invariance (cf. \cite{ThomsenUnivProp}). The exterior tensor product can be also defined in the equivariant setting, say $\widetilde{\tau}^G_{\mathbb{C}}$; the action of $G$ on a tensor product of C$^*$-algebras being the diagonal action. Again, $\mathscr{KK}^{G}$ is equipped in this way with a monoidal structure and the analogue to Equation (\ref{eq.TensorProdKasparovProd}) holds. Furthermore, if $\mathbb{G}$ is a regular locally compact quantum group, then $\mathscr{KK}^{\mathbb{G}}$ is also characterised by a universal property as in the classical setting (cf. \cite{VoigtPoincareDuality}). If $D(\mathbb{G})$ denotes the Drinfeld double of $\mathbb{G}$, then $\mathscr{KK}^{D(\mathbb{G})}$ is equipped with a monoidal structure by means of the \emph{braided tensor product} (cf. \cite{VoigtPoincareDuality}). Again, the analogue to Equation (\ref{eq.TensorProdKasparovProd}) holds.
	\end{rem}

	Assume for the moment that $\widehat{\mathbb{G}}$ is \emph{torsion-free}. In that case, consider the usual complementary pair of localizing subcategories in $\mathscr{KK}^{\widehat{\mathbb{G}}}$, $(\mathscr{L}_{\widehat{\mathbb{G}}}, \mathscr{N}_{\widehat{\mathbb{G}}})$. Denote by $(L,N)$ the canonical triangulated functors associated to this complementary pair. More precisely we have that $\mathscr{L}_{\widehat{\mathbb{G}}}$ is defined as the \emph{localizing subcategory of $\mathscr{KK}^{\widehat{\mathbb{G}}}$ generated by the objects of the form $\text{Ind}^{\widehat{\mathbb{G}}}_{\mathbb{E}}(C)=C\otimes c_0(\widehat{\mathbb{G}})$ with $C$ any C$^*$-algebra in the Kasparov category $\mathscr{KK}$} and $\mathscr{N}_{\widehat{\mathbb{G}}}$ is defined as the \emph{localizing subcategory of objects which are isomorphic to $0$ in $\mathscr{KK}$}: $\mathscr{L}_{\widehat{\mathbb{G}}}:=\langle\{\text{Ind}^{\widehat{\mathbb{G}}}_{\mathbb{E}}(C)=C\otimes c_0(\widehat{\mathbb{G}})\ |\ C\in \text{Obj}(\mathscr{KK})\}\rangle$ and $\mathscr{N}_{\widehat{\mathbb{G}}}=\{A\in \text{Obj}(\mathscr{KK}^{\widehat{\mathbb{G}}})\ |\ \text{Res}^{\widehat{\mathbb{G}}}_{\mathbb{E}}(A)\cong 0\}$.
	
	If $\widehat{\mathbb{G}}$ is \emph{not} torsion-free, then a technical property lacked in the literature in order to define a suitable complementary pair. The natural candidate used in the related works (see for instance \cite{MeyerNestTorsion} and \cite{VoigtBaumConnesAutomorphisms}) is given by the following localizing subcategories of $\mathscr{KK}^{\widehat{\mathbb{G}}}$:
	$$\mathscr{L}_{\widehat{\mathbb{G}}}:=\langle\{C\otimes T\underset{r}{\rtimes}\mathbb{G}\ |\  C\in \text{Obj}(\mathscr{KK}),\ T\in\mathbb{T}\mbox{or}(\widehat{\mathbb{G}})\}\rangle,$$
	$$\mathscr{N}_{\widehat{\mathbb{G}}}:=\mathscr{L}^{\dashv}_{\widehat{\mathbb{G}}}=\{A\in \text{Obj}(\mathscr{KK}^{\widehat{\mathbb{G}}})\ |\ KK^{\widehat{\mathbb{G}}}(L, A)\cong 0\mbox{, $\forall$ $L\in \text{Obj}(\mathscr{L}_{\widehat{\mathbb{G}}})$}\}.$$
	
	\begin{rem}
			We put $\widehat{\mathscr{L}}_{\widehat{\mathbb{G}}}:=\langle\{C\otimes T\ |\  C\in \text{Obj}(\mathscr{KK}),\ T\in\mathbb{T}\mbox{or}(\widehat{\mathbb{G}})\}\rangle$, so that we have $\widehat{\mathscr{L}}_{\widehat{\mathbb{G}}} \rtimes \mathbb{G}=\mathscr{L}_{\widehat{\mathbb{G}}}$ by definition. Similarly we put $\widehat{\mathscr{N}}_{\widehat{\mathbb{G}}}:=\mathscr{N}_{\widehat{\mathbb{G}}}\rtimes \widehat{\mathbb{G}}$. The pair $(\widehat{\mathscr{L}}_{\widehat{\mathbb{G}}}, \widehat{\mathscr{N}}_{\widehat{\mathbb{G}}})$ is still complementary. We denote by $(\widehat{L}, \widehat{N})$ the corresponding triangulated functors defining the $(\widehat{\mathscr{L}}_{\widehat{\mathbb{G}}}, \widehat{\mathscr{N}}_{\widehat{\mathbb{G}}})$-triangles in $\mathscr{KK}^{\mathbb{G}}$.
	\end{rem}
	
	Recently, Y. Arano and A. Skalski \cite{YukiBCTorsion} have showed that these two subcategories form indeed a complementary pair of localizing subcategories in $\mathscr{KK}^{\widehat{\mathbb{G}}}$. Also, the author in collaboration with K. De Commer and R. Nest has obtained the same conclusion in \cite{KennyNestRubenBCProjective} for permutation torsion-free discrete quantum groups as an application of the study of the projective representation theory for compact quantum groups. In either case, the complementarity of the pair $(\mathscr{L}_{\widehat{\mathbb{G}}}, \mathscr{N}_{\widehat{\mathbb{G}}})$ is based in a generalisation of the \emph{Green-Julg isomorphism}. More precisely, given such a torsion action of $\mathbb{G}$, say $(T, \delta)$, we define the following triangulated functors:
	$$
		\begin{array}{rcclccl}
			j_{\mathbb{G}, T}:&\mathscr{KK}^\mathbb{G}& \longrightarrow & \mathscr{KK}, &(B, \beta)& \longmapsto & j_{\mathbb{G}, T}(B,\beta):= B\underset{r, \beta}{\rtimes}\mathbb{G}\underset{r, \overline{\delta}}{\ltimes}T^{op},
		\end{array}
	$$
	$$
		\begin{array}{rcclccl}
			\tau_{ T}:&\mathscr{KK}& \longrightarrow & \mathscr{KK}^\mathbb{G}, &C& \longmapsto & \tau_{T}(C):= (C\otimes T, id\otimes \delta).
		\end{array}
	$$
	
	The proof of the following result can be found in \cite{YukiBCTorsion} (see also \cite{KennyNestRubenBCProjective} for a different approach).
	\begin{theoSec}
		Let $\mathbb{G}$ be a compact quantum group. Let $(T, \delta)$ be a torsion action of $\mathbb{G}$. Then $\tau_T: \mathscr{KK}\longrightarrow \mathscr{KK}^{\mathbb{G}}$ is a left adjoint of $j_{\mathbb{G}, T}: \mathscr{KK}^{\mathbb{G}}\longrightarrow \mathscr{KK}$ as triangulated functors. More precisely,
		$$\Psi_T: KK^\mathbb{G}(C\otimes T, B)\overset{\sim}{\longrightarrow} KK\big(C, B\underset{r, \beta}{\rtimes}\mathbb{G}\underset{r, \overline{\delta}}{\ltimes}T^{op}\big),$$
		for all $C\in \text{Obj}(\mathscr{KK})$ and $(B, \beta)\in \text{Obj}(\mathscr{KK}^\mathbb{G})$.
	\end{theoSec}
	
	Then the general Meyer-Nest's machinery yields in particular that $\widehat{\mathscr{N}}_{\widehat{\mathbb{G}}}=\ker_{\text{Obj}}\big((j_{\mathbb{G}, T})_{T\in \text{Tor}(\widehat{\mathbb{G}})}\big)$. This allows to define a quantum assembly map for arbitrary discrete quantum groups (not necessarily torsion-free) and thus a quantum version of the (strong) BC property. More precisely, consider the following homological functor:
	$$
		\begin{array}{rcclccl}
			F_*:&\mathscr{KK}^{\widehat{\mathbb{G}}}& \longrightarrow &\mathscr{A}b^{\mathbb{Z}/2}, &(A,\alpha) & \longmapsto &F_*(A):=K_{*}(A\underset{\alpha, r}{\rtimes} \widehat{\mathbb{G}}).
		\end{array}
	$$
	
	The quantum assembly map for $\widehat{\mathbb{G}}$ is given by the natural transformation $\eta^{\widehat{\mathbb{G}}}: \mathbb{L}F_*\longrightarrow F_*$, where $\mathbb{L}F_*$ denotes the localisation of $F_*$ with respect to the complementary pair $(\mathscr{L}_{\widehat{\mathbb{G}}}, \mathscr{N}_{\widehat{\mathbb{G}}})$. More precisely, given any $\widehat{\mathbb{G}}$-C$^*$-algebra $(A, \alpha)$, we consider a $(\mathscr{L}_{\widehat{\mathbb{G}}}, \mathscr{N}_{\widehat{\mathbb{G}}})$-triangle associated to $A$, say $\Sigma N(A)\longrightarrow L(A)\overset{u}{\longrightarrow} A\longrightarrow N(A)$. Then $\eta^{\widehat{\mathbb{G}}}_A=F_*(u)$. Let us point out a more precise picture for $\eta^{\widehat{\mathbb{G}}}_A$. On the one hand, given the $(\mathscr{L}_{\widehat{\mathbb{G}}}, \mathscr{N}_{\widehat{\mathbb{G}}})$-triangle $\Sigma N(A)\longrightarrow L(A)\overset{u}{\longrightarrow} A\longrightarrow N(A)$, the map $L(A)\underset{r}{\rtimes} \widehat{\mathbb{G}}\overset{u}{\rightarrow} A\underset{r}{\rtimes} \widehat{\mathbb{G}}$ must be viewed as a Kaparov triple $u\underset{r}{\rtimes} \widehat{\mathbb{G}}\in KK(L(A)\underset{r}{\rtimes} \widehat{\mathbb{G}}, A\underset{r}{\rtimes} \widehat{\mathbb{G}})$, i.e. a homomorphism in $\mathscr{KK}$. On the other hand, the functor $F_*(\ \cdot\ )=KK_*(\mathbb{C}, (\ \cdot\ )\underset{r}{\rtimes} \widehat{\mathbb{G}})$ can be viewed as the covariant homomorphism functor $KK(\mathbb{C}, \ \cdot\ )$ in $\mathscr{KK}$, which consists in composing, i.e. making a Kasparov product. In other words, when we apply this to crossed product C$^*$-algebras we have:
	\begin{equation}\label{eq.PictureQAssemblyMap}
	\begin{split}
	\eta^{\widehat{\mathbb{G}}}_A=F_*(u)=KK_*(\mathbb{C}, u\underset{r}{\rtimes} \widehat{\mathbb{G}})=(\ \cdot\ )\underset{L(A)\underset{r}{\rtimes} \widehat{\mathbb{G}}}{\otimes} u\underset{r}{\rtimes} \widehat{\mathbb{G}}: KK(\mathbb{C}, L(A)\underset{r}{\rtimes} \widehat{\mathbb{G}})\rightarrow KK(\mathbb{C}, A\underset{r}{\rtimes} \widehat{\mathbb{G}}).
	\end{split}
	\end{equation}
	\begin{defi}
		Let $\widehat{\mathbb{G}}$ be a discrete quantum group. We say that $\widehat{\mathbb{G}}$ satisfies the quantum Baum-Connes property (with coefficients) if the natural transformation $\eta^{\widehat{\mathbb{G}}}: \mathbb{L}F_*\longrightarrow F_*$ is a natural equivalence. We say that $\widehat{\mathbb{G}}$ satisfies the \emph{strong} quantum Baum-Connes property if $\mathscr{KK}^{\widehat{\mathbb{G}}}=\mathscr{L}_{\widehat{\mathbb{G}}}$. We abbreviate the predicate \emph{Baum-Connes property} by \emph{BC property}. 
	\end{defi}
	
	\begin{note}\label{note.DiracHom}
		The following nomenclature is useful. Given $A\in \text{Obj}(\mathscr{KK}^{\widehat{\mathbb{G}}})$ consider a $(\mathscr{L}_{\widehat{\mathbb{G}}}, \mathscr{N}_{\widehat{\mathbb{G}}})$-triangle associated to $A$, say $\Sigma N(A)\longrightarrow L(A)\overset{u}{\longrightarrow} A\longrightarrow N(A)$. We know that such triangles are distinguished and unique up to isomorphism (cf. \cite{MeyerNest} for a proof). The homomorphism $u:L(A)\longrightarrow A$ is called \emph{Dirac homomorphism for $A$}. Sometimes we write $D:=u$. In particular, we consider the Dirac homomorphism for $\mathbb{C}$ (as trivial $\widehat{\mathbb{G}}$-C$^*$-algebra), $D_{\mathbb{C}}:L(\mathbb{C})\longrightarrow \mathbb{C}$. We refer to $D_{\mathbb{C}}$ simply as \emph{Dirac homomorphism}.
	\end{note}

\section{\textsc{The Künneth classes}}\label{sec.KunnethFunctors}
	We start by recalling some terminology introduced in \cite{ChabertEchterhoffOyono} in the context of the (abstract) Künneth theorem.
	\begin{defiSec}\label{defi.KunnethCategory}
		A Künneth category is a subcategory $\mathcal{S}$ of $\text{C}^*\text{-Alg}$ such that:
		\begin{enumerate}[i)]
			\item $\mathcal{S}$ contains all separable commutative C$^*$-algebras.
			\item $\mathcal{S}$ is closed under stabilization, i.e. if $B\in \text{Obj}(\mathcal{S})$, then $\mathcal{K}\otimes B\in \text{Obj}(\mathcal{S})$.
			\item $\mathcal{S}$ is closed under suspension, i.e. if $B\in \text{Obj}(\mathcal{S})$, then $\Sigma(B)\in \text{Obj}(\mathcal{S})$ are objects in $\mathcal{S}$.
			\item If $0\rightarrow J\rightarrow B\rightarrow B/J\rightarrow 0$ is a semi-split short exact sequence in $\mathcal{S}$ such that two of the C$^*$-algebras are in $\mathcal{S}$, the so is the third.
		\end{enumerate}
	\end{defiSec}
	
	\begin{defiSec}\label{defi.KunnethFunctor}
		Let $\mathcal{S}$ be a Künneth category. A Künneth functor on $\mathcal{S}$ is an additive convariant functor $\kappa_*: \mathcal{S}\longrightarrow \mathscr{A}\text{b}^{\mathbb{Z}/2}$ such that:
		\begin{enumerate}[i)]
			\item $\kappa_*$ is invariant under stabilization and under homotopy.
			\item If $0\rightarrow J\rightarrow B\rightarrow B/J\rightarrow 0$ is a semi-split short exact sequence in $\mathcal{S}$, then the sequence $\kappa_*(J)\rightarrow \kappa_*(B)\rightarrow \kappa_*(B/J)$ is exact.
			\item $\kappa_*$ is stable, i.e. $\kappa_*(\Sigma(B))\cong \kappa_{*+1}(B)$ naturally, for all $B\in \text{Obj}(\mathcal{S})$.
			\item There exists a natural zero-graded homomorphism $\alpha: \kappa_*(\mathbb{C})\otimes K_*(B)\rightarrow \kappa_*(B)$, for all $B\in \text{Obj}(\mathcal{S})$ such that $\alpha$ is an isomorphism whenever $K_*(B)$ is free abelian.
		\end{enumerate}
	\end{defiSec}
	
	\begin{theoSec}\label{theo.AbstractKunnethTheorem}
		Let $\mathcal{S}$ be a Künneth category. If $\kappa_*$ is a Künneth functor on $\mathcal{S}$, then there exists a natural $1$-graded homomorphism $\beta: \kappa_*(B)\rightarrow \text{Tor}(\kappa_*(\mathbb{C}), K_*(B))$, for all $B\in \text{Obj}(\mathcal{S})$ such that the sequence:
		$$0\longrightarrow\kappa_*(\mathbb{C})\otimes K_*(B)\overset{\alpha}{\longrightarrow} \kappa_*(B)\overset{\beta}{\longrightarrow} \text{Tor}(\kappa_*(\mathbb{C}), K_*(B))\longrightarrow 0$$
		is exact. This sequence is called \emph{Künneth sequence for $\kappa_*$}.
	\end{theoSec}
	
	\begin{exsSec}\label{exs.KunnethClasses}
		\begin{enumerate}
			\item Let $A$ and $B$ be two C$^*$-algebras. The external Kasparov product: $$KK(\mathbb{C}, A)\times KK(\mathbb{C}, B)\overset{\widetilde{\tau}_{\mathbb{C}}}{\longrightarrow} KK(\mathbb{C}, A\otimes B)$$
			yields a group homomorphism $\alpha: K_*(A)\otimes K_*(B)\rightarrow K_*(A\otimes B)$, which is natural on $A$ and $B$. We denote by $\mathcal{N}$ the class of all C$^*$-algebras $A$ in $\text{C}^*\text{-Alg}$ such that $\alpha$ is an isomorphism for all separable C$^*$-algebra $B$ with $K_*(B)$ free abelian. By abuse of notation, we still denote by $\mathcal{N}$ the full subcategory of $\text{C}^*\text{-Alg}$ generated by the class $\mathcal{N}$. It is well-known that $\mathcal{N}$ is a Künneth category in the sense of Definition \ref{defi.KunnethCategory} (cf. \cite{ChabertEchterhoffOyono}). Moreover, if $A\in\mathcal{N}$, it is clear that the map $B\mapsto K_*(A\otimes B)$ defines a functor on the whole category $\text{C}^*\text{-Alg}$ satisfying properties $(i)-(iv)$ from Definition \ref{defi.KunnethFunctor} (cf. \cite{ChabertEchterhoffOyono}). In other words, this map defines a Künneth functor. Consequently, Theorem \ref{theo.AbstractKunnethTheorem} yields that the Künneth sequence:
			$$0\longrightarrow K_*(A)\otimes K_*(B)\overset{\alpha}{\longrightarrow} K_*(A\otimes B)\overset{\beta}{\longrightarrow} \text{Tor}(K_*(A), K_*(B))\longrightarrow 0$$
			is exact for all $A\in\mathcal{N}$ and all $B\in \text{Obj}(\text{C}^*\text{-Alg})$. It is important to mention that the class $\mathcal{N}$ contains the bootstrap class and it satisfies the following remarkable stability properties (cf. \cite[Lemma 4.4]{ChabertEchterhoffOyono}):
			\begin{lemSec}\label{lem.StabilityClassN}
				$\mathcal{N}$ contains the bootstrap class and the following properties hold.
				\begin{enumerate}[i)]
					\item $\mathcal{N}$ is stable under $KK$-equivalence hence it is stable under Morita equivalence. 
					\item If $0\rightarrow I\rightarrow A\rightarrow A/I\rightarrow$ is a semi-split short exact sequence of C$^*$-algebras such that two of them are in $\mathcal{N}$, then so is the third.
					\item If $A, B\in\mathcal{N}$, then $A\otimes B\in\mathcal{N}$.
					\item If $A=\underset{\rightarrow}{\lim}\ A_i$ such that all structure maps are injective and $A_i\in\mathcal{N}$ for all $i$, then $A\in\mathcal{N}$.
				\end{enumerate}
			\end{lemSec}
			\item Let $G$ be a locally compact group. Let $A$ be a $G$-C$^*$-algebra and $B$ a C$^*$-algebra. Observe that one has an obvious map $KK(\mathbb{C}, B)\rightarrow KK^G(\mathbb{C}, B)$, which consists in equipping with the trivial action of $G$. Then, for all $G$-compact subspace $X\subset \underline{E}G$, the external Kasparov product: 
			$$KK^G(C_0(X), A)\times KK(\mathbb{C}, B)\longrightarrow KK^G(C_0(X), A)\times KK^G(\mathbb{C}, B)\overset{\widetilde{\tau}_{\mathbb{C}}}{\longrightarrow} KK(C_0(X), A\otimes B)$$
			yields a group homomorphism $\alpha_X: KK^G(C_0(X), A)\otimes K_*(B)\rightarrow KK^G(C_0(X), A\otimes B)$, for all $G$-compact subspace $X\subset \underline{E}G$, which is natural on $A$ and $B$. One checks that the maps $\{\alpha_X\}_{\underset{\text{$G$-compact}}{X\subset \underline{E}G}}$ are compatible with the inclusion of $G$-compact spaces. Hence, one obtains a group homomorphism $\alpha_G: K^{\text{top}}_*(G; A)\otimes K_*(B)\rightarrow K^{\text{top}}_*(G; A\otimes B)$, which is natural on $A$ and $B$. We denote by $\mathcal{N}_G$ the class of all $G$-C$^*$-algebras $A$ in $G$-$\text{C}^*\text{-Alg}$ such that $\alpha_G$ is an isomorphism for all separable C$^*$-algebra $B$ with $K_*(B)$ free abelian. By abuse of notation, we still denote by $\mathcal{N}_G$ the full subcategory of $G$-$\text{C}^*\text{-Alg}$ generated by the class $\mathcal{N}_G$. It is well-known that $\mathcal{N}_G$ is a Künneth category in the sense of Definition \ref{defi.KunnethCategory} (cf. \cite{ChabertEchterhoffOyono}). Moreover, if $A\in\mathcal{N}_G$, it is clear that the map $B\mapsto K^{\text{top}}_*(G; A\otimes B)$ defines a functor on the whole category $\text{C}^*\text{-Alg}$ satisfying properties $(i)-(iv)$ from Definition \ref{defi.KunnethFunctor} (cf. \cite{ChabertEchterhoffOyono}). In other words, this map defines a Künneth functor. Consequently, Theorem \ref{theo.AbstractKunnethTheorem} yields that the Künneth sequence:
			$$0\longrightarrow K^{\text{top}}_*(G; A)\otimes K_*(B)\overset{\alpha_G}{\longrightarrow} K^{\text{top}}_*(G; A\otimes B)\overset{\beta_G}{\longrightarrow} \text{Tor}(K^{\text{top}}_*(G; A), K_*(B))\longrightarrow 0$$
			is exact for all $A\in\mathcal{N}_G$ and all $B\in \text{Obj}(\text{C}^*\text{-Alg})$.
		\end{enumerate}
	\end{exsSec}

	\subsection{Preparatory observations}\label{sec.PrepKunneth}
	In order to provide examples of Künneth functors in the framework of quantum groups, let us recall some constructions appearing in \cite{RubenSemiDirect}. Let $\mathbb{G}$ and $\mathbb{H}$ be two compact quantum groups. We put $\mathbb{F}:=\mathbb{G}\times\mathbb{H}$ for the corresponding quantum direct product. Consider the complementary pair of localizing subcategories in $\mathscr{KK}^{\widehat{\mathbb{F}}}$, $\mathscr{KK}^{\widehat{\mathbb{G}}}$ and $\mathscr{KK}^{\widehat{\mathbb{H}}}$: $(\mathscr{L}_{\widehat{\mathbb{F}}}, \mathscr{N}_{\widehat{\mathbb{F}}})$, $(\mathscr{L}_{\widehat{\mathbb{G}}}, \mathscr{N}_{\widehat{\mathbb{G}}})$ and $(\mathscr{L}_{\widehat{\mathbb{H}}}, \mathscr{N}_{\widehat{\mathbb{H}}})$, respectively; as explained in Section \ref{sec.QuantumBC}. Accordingly, the canonical triangulated functors associated to these complementary pairs are denoted by $(L,N)$, $(L', N')$ and $(L'', N'')$, respectively. Consider the homological functors defining the \emph{quantum} Baum-Connes assembly maps for $\widehat{\mathbb{F}}$, $\widehat{\mathbb{G}}$ and $\widehat{\mathbb{H}}$:
	$$F:\mathscr{KK}^{\widehat{\mathbb{F}}} \rightarrow \mathscr{A}b^{\mathbb{Z}/2}\mbox{, } (C,\delta) \mapsto F(C):=K_{*}(C\underset{\delta, r}{\rtimes} \widehat{\mathbb{F}}),$$
	$$F':\mathscr{KK}^{\widehat{\mathbb{G}}} \rightarrow \mathscr{A}b^{\mathbb{Z}/2}\mbox{, } (A,\alpha) \mapsto F'(A):=K_{*}(A\underset{\alpha,r}{\rtimes} \widehat{\mathbb{G}}),$$
	$$F'':\mathscr{KK}^{\widehat{\mathbb{H}}} \rightarrow \mathscr{A}b^{\mathbb{Z}/2}\mbox{, } (B,\beta) \mapsto F''(B):= K_{*}(B \underset{\beta,r}{\rtimes} \widehat{\mathbb{H}}).$$
	
	Therefore, the quantum assembly maps for $\widehat{\mathbb{F}}$, $\widehat{\mathbb{G}}$ and $\widehat{\mathbb{H}}$ are given by the natural transformations $\mathbb{L}F\overset{\eta^{\widehat{\mathbb{F}}}}{\longrightarrow} F$, $\mathbb{L}F'\overset{\eta^{\widehat{\mathbb{G}}}}{\longrightarrow} F'$ and $\mathbb{L}F''\overset{\eta^{\widehat{\mathbb{H}}}}{\longrightarrow} F''$, respectively. Next, consider the functor: 
	$$\mathcal{Z}:\mathscr{KK}^{\widehat{\mathbb{G}}}\times \mathscr{KK}^{\widehat{\mathbb{H}}} \rightarrow \mathscr{KK}^{\widehat{\mathbb{F}}}$$ 
	defined on objects by $(A,\alpha)\times (B,\beta) \mapsto \mathcal{Z}(A, B):= (C:=A\otimes B, \delta:=\alpha\otimes \beta)$ and on homomorphisms by $\mathcal{Z}(\mathcal{X},\mathcal{Y}):=\mathcal{X}\otimes \mathcal{Y}:=\tau_{B}(\mathcal{X})\underset{A'\otimes B}{\otimes}{_{A'}}\tau(\mathcal{Y})$, for all $\mathcal{X}\in KK^{\widehat{\mathbb{G}}}(A,A')$, $\mathcal{Y}\in KK^{\widehat{\mathbb{H}}}(B,B')$ with $(A, \alpha)\times (B, \beta), (A',\alpha')\times (B',\beta')\in \text{Obj}(\mathscr{KK}^{\widehat{\mathbb{G}}})\times \text{Obj}(\mathscr{KK}^{\widehat{\mathbb{H}}})$. 

	\begin{lem}\label{lem.FunctorPreservingLDirectProduct}
		Let $\mathbb{G}$, $\mathbb{H}$ be two compact quantum groups and let $\mathbb{F}:=\mathbb{G}\times \mathbb{H}$ be the corresponding quantum direct product of $\mathbb{G}$ and $\mathbb{H}$. The functor $\mathcal{Z}:\mathscr{KK}^{\widehat{\mathbb{G}}}\times \mathscr{KK}^{\widehat{\mathbb{H}}} \rightarrow \mathscr{KK}^{\widehat{\mathbb{F}}}$ is such that $\mathcal{Z}(\mathscr{L}_{\widehat{\mathbb{G}}}\times \mathscr{L}_{\widehat{\mathbb{H}}})\subset \mathscr{L}_{\widehat{\mathbb{F}}}$. If, moreover, either $\widehat{\mathbb{G}}$ or $\widehat{\mathbb{H}}$ satisfies the strong quantum BC property, then $\mathcal{Z}(\mathscr{N}_{\widehat{\mathbb{G}}}\times \mathscr{N}_{\widehat{\mathbb{H}}})\subset \mathscr{N}_{\widehat{\mathbb{F}}}$. Besides, we have the following.
		\begin{enumerate}[i)]
		
		\item If $(A_0, \alpha_0)\in \text{Obj}(\mathscr{KK}^{\widehat{\mathbb{G}}})$ is a given $\widehat{\mathbb{G}}$-C$^*$-algebra, the functor ${_{A_0}}\mathcal{Z}:\mathscr{KK}^{\widehat{\mathbb{H}}} \rightarrow \mathscr{KK}^{\widehat{\mathbb{F}}}\mbox{, }(B,\beta)  \mapsto {_{A_0}}\mathcal{Z}(B):= \mathcal{Z}(A_0, B)$ is triangulated. If, moreover, we assume that $\widehat{\mathbb{H}}$ satisfies the strong quantum BC property, then ${_{A_0}}\mathcal{Z}(\mathscr{N}_{\widehat{\mathbb{H}}})\subset \mathscr{N}_{\widehat{\mathbb{F}}}$.
		
		\item If $(B_0, \beta_0)\in \text{Obj}(\mathscr{KK}^{\widehat{\mathbb{H}}})$ is a given $\widehat{\mathbb{H}}$-C$^*$-algebra, the functor $\mathcal{Z}_{B_0}:\mathscr{KK}^{\widehat{\mathbb{G}}} \rightarrow\mathscr{KK}^{\widehat{\mathbb{F}}}\mbox{, }(A,\alpha) \mapsto \mathcal{Z}_{B_0}(A):= \mathcal{Z}(A, B_0)$ is triangulated. If, moreover, we assume that $\widehat{\mathbb{G}}$ satisfies the strong quantum BC property, then $\mathcal{Z}_{B_0}(\mathscr{N}_{\widehat{\mathbb{G}}})\subset \mathscr{N}_{\widehat{\mathbb{F}}}$.
		\end{enumerate}
	\end{lem}
	\begin{proof}
		Most part of the argument used in \cite[Lemma 5.2.1]{RubenSemiDirect} applies. But here we have removed the torsion-freeness assumption, so we need to modify some steps of the proof of \cite[Lemma 5.2.1]{RubenSemiDirect}. First, given an object $(T,\tau)\times (S,\sigma)\in \text{Obj}(\mathscr{KK}^{\mathbb{G}})\times \text{Obj}(\mathscr{KK}^{\mathbb{H}})$ with $(T,\tau)$ a torsion action of $\mathbb{G}$ and $(S, \sigma)$ a torsion action of $\mathbb{H}$, then $(R:=T\otimes S, \rho:=\tau\otimes \sigma)$ is a torsion action of $\mathbb{F}$ as observed in Section \ref{sec.QDP}. Next, using Proposition \ref{pro.DirectProductTensorProduct} we write:
		\begin{equation*}
			\begin{split}
				\mathcal{Z}\big(T\underset{\tau, r}{\rtimes}\mathbb{G}\otimes C_1, S\underset{\sigma, r}{\rtimes}\mathbb{H}\otimes C_2\big)&=T\underset{\tau, r}{\rtimes}\mathbb{G}\otimes C_1\otimes S\underset{\sigma, r}{\rtimes}\mathbb{H}\otimes C_2\\
				&\cong T\underset{\tau, r}{\rtimes}\mathbb{G}\otimes S\underset{\sigma, r}{\rtimes}\mathbb{H}\otimes C_1\otimes C_2\cong R\underset{\rho, r}{\rtimes} \mathbb{F}\otimes C_3\mbox{,}
			\end{split}
		\end{equation*}
		where $C_3:=C_1\otimes C_2\in \text{Obj}(\mathscr{KK})$. Consequently, $\mathcal{Z}(\mathscr{L}_{\widehat{\mathbb{G}}}\times \mathscr{L}_{\widehat{\mathbb{H}}})\subset \mathscr{L}_{\widehat{\mathbb{F}}}$ as in \cite[Lemma 5.2.1]{RubenSemiDirect}. 
		
		Now, if either $\widehat{\mathbb{G}}$ or $\widehat{\mathbb{H}}$ satisfies the strong quantum BC property, then we have by definition that either $\mathscr{KK}^{\widehat{\mathbb{G}}}=\mathscr{L}_{\widehat{\mathbb{G}}}$ or $\mathscr{KK}^{\widehat{\mathbb{H}}}=\mathscr{L}_{\widehat{\mathbb{H}}}$. Since $(\mathscr{L}_{\widehat{\mathbb{G}}}, \mathscr{N}_{\widehat{\mathbb{G}}})$ and $(\mathscr{L}_{\widehat{\mathbb{H}}}, \mathscr{N}_{\widehat{\mathbb{H}}})$ are complementary pairs as explained in Section \ref{sec.QuantumBC}, the latter means that we have either $\mathscr{N}_{\widehat{\mathbb{G}}}=0$ or $\mathscr{N}_{\widehat{\mathbb{H}}}=0$ hence either $\widehat{\mathscr{N}}_{\widehat{\mathbb{G}}}=0$ or $\widehat{\mathscr{N}}_{\widehat{\mathbb{H}}}=0$. In either case, given $(A, \alpha)\in \text{Obj}(\widehat{\mathscr{N}}_{\widehat{\mathbb{G}}})$ and $(B, \beta)\in \text{Obj}(\widehat{\mathscr{N}}_{\widehat{\mathbb{H}}})$ we have $j_{\mathbb{F}, R}(A\otimes B)=(A\otimes B)\underset{r, \alpha\otimes \beta}{\rtimes}\mathbb{F}\underset{r, \overline{\rho}}{\ltimes}R^{op}\cong 0$ in $\mathscr{KK}$, for every torsion action $(R, \rho)$ of $\mathbb{F}$. This means, as explained in Section \ref{sec.QuantumBC}, that $(A\otimes B, \alpha\otimes\beta)\in \text{Obj}(\widehat{\mathscr{N}}_{\widehat{\mathbb{F}}})$. In other words, $\mathcal{Z}(\widehat{\mathscr{N}}_{\widehat{\mathbb{G}}}\times \widehat{\mathscr{N}}_{\widehat{\mathbb{H}}})\subset \widehat{\mathscr{N}}_{\widehat{\mathbb{F}}}$, which yields $\mathcal{Z}(\mathscr{N}_{\widehat{\mathbb{G}}}\times \mathscr{N}_{\widehat{\mathbb{H}}})\subset \mathscr{N}_{\widehat{\mathbb{F}}}$ by virtue of Baaj-Skandalis duality and Proposition \ref{pro.DirectProductTensorProduct}.
				
		To conclude, let us show property $(i)$ of the statement (the proof of property $(ii)$ is similar). It remains to show that, given a $\widehat{\mathbb{G}}$-C$^*$-algebra $(A_0, \alpha_0)\in \text{Obj}(\mathscr{KK}^{\widehat{\mathbb{G}}})$, then the functor ${_{A_0}}\mathcal{Z}$ is triangulated. For this, the analogous argument as in \cite[Lemma 5.2.1]{RubenSemiDirect} applies.
	\end{proof}
	
	The previous lemma allows to improve \cite[Lemma 5.2.2]{RubenSemiDirect} by removing the torsion-freeness assumption as follows.
	\begin{lem}\label{lem.InvertibleElementDirectProduct}
		Let $\mathbb{G}$, $\mathbb{H}$ be two compact quantum groups and let $\mathbb{F}:=\mathbb{G}\times \mathbb{H}$ be the corresponding quantum direct product of $\mathbb{G}$ and $\mathbb{H}$.
		
		\begin{enumerate}[i)]
			\item For all $\widehat{\mathbb{G}}$-C$^*$-algebra $(A,\alpha)$ and all $\widehat{\mathbb{H}}$-C$^*$-algebra $(B,\beta)$ there exists a Kasparov triple $\psi\in KK^{\widehat{\mathbb{F}}}\big(L'(A)\otimes L''(B), L(A\otimes B)\big)$ such that the diagram:
		\begin{equation*}\label{eq.CommutativeDiagramDirectProduct}
		\begin{gathered}
			\xymatrix@C=20mm@!R=15mm{
				\mbox{$L'(A)\underset{r}{\rtimes}\widehat{\mathbb{G}}\otimes L''(B) \underset{r}{\rtimes}\widehat{\mathbb{H}}$}\ar[d]_{\mbox{$ u'\underset{r}{\rtimes}\widehat{\mathbb{G}}\otimes u''\underset{r}{\rtimes}\widehat{\mathbb{H}}$}}\ar[r]^-{\mbox{$\Psi$}}&\mbox{$L(A\otimes B) \underset{r}{\rtimes}\widehat{\mathbb{F}}$}\ar[d]^{\mbox{$u \underset{r}{\rtimes}\widehat{\mathbb{F}}$}}\\
				\mbox{$A\underset{r}{\rtimes}\widehat{\mathbb{G}}\otimes B\underset{r}{\rtimes}\widehat{\mathbb{H}}$}\ar[r]^{\mbox{$\cong$}}&\mbox{$(A\otimes B)\underset{r}{\rtimes}\widehat{\mathbb{F}}$}}
		\end{gathered}
		\end{equation*}
		is commutative where $\Psi:=\psi\underset{r}{\rtimes}\widehat{\mathbb{F}}$ and $u'$, $u''$, $u$ are the Dirac homomorphisms for $A$, $B$, $A\otimes B$, respectively.
			\item Assume that $\widehat{\mathbb{H}}$ satisfies the strong quantum BC property. For all $\widehat{\mathbb{G}}$-C$^*$-algebra $(A_0,\alpha_0)\in \mathscr{L}_{\widehat{\mathbb{G}}}$ and all $\widehat{\mathbb{H}}$-C$^*$-algebra $(B,\beta)$ there exists an \emph{invertible} Kasparov triple ${_{A_0}}\psi\in KK^{\widehat{\mathbb{F}}}\big(A_0\otimes L''(B), L(A_0\otimes B)\big)$ such that the diagram: 
		\begin{equation*}\label{eq.CommutativeDiagramDirectProductbisH}
		\begin{gathered}
			\xymatrix@C=20mm@!R=15mm{
				\mbox{$A_0\underset{r}{\rtimes}\widehat{\mathbb{G}}\otimes L''(B)\underset{r}{\rtimes}\widehat{\mathbb{H}}$}\ar[d]_{\mbox{${_{A_0}}\mathcal{Z}(u'')\underset{r}{\rtimes}\widehat{\mathbb{F}}$}}\ar[r]^-{\mbox{$\underset{\sim}{{_{A_0}}\Psi}$}}&\mbox{$L(A_0\otimes B)\underset{r}{\rtimes}\widehat{\mathbb{F}}$}\ar[d]^{\mbox{$u\underset{r}{\rtimes}\widehat{\mathbb{F}}$}}\\
				\mbox{$A_0\underset{r}{\rtimes}\widehat{\mathbb{G}}\otimes B\underset{r}{\rtimes}\widehat{\mathbb{H}}$}\ar[r]^{\mbox{$\cong$}}&\mbox{$(A_0\otimes B)\underset{r}{\rtimes}\widehat{\mathbb{F}}$}}
		\end{gathered}
		\end{equation*}
		is commutative where ${_{A_0}}\Psi:={_{A_0}}\psi\underset{r}{\rtimes}\widehat{\mathbb{F}}$ and $u''$, $u$ are the Dirac homomorphism for $B$, $A_0\otimes B$, respectively.
		\item Assume that $\widehat{\mathbb{G}}$ satisfies the strong quantum BC property. For all $\widehat{\mathbb{H}}$-C$^*$-algebra $(B_0,\beta_0)\in \mathscr{L}_{\widehat{\mathbb{H}}}$ and all $\widehat{\mathbb{G}}$-C$^*$-algebra $(A,\alpha)$ there exists an \emph{invertible} Kasparov triple $\psi_{B_0}\in KK^{\widehat{\mathbb{F}}}\big(L'(A)\otimes B_0, L(A\otimes B_0)\big)$ such that the diagram: 
		\begin{equation*}\label{eq.CommutativeDiagramDirectProductbisG}
		\begin{gathered}
			\xymatrix@C=20mm@!R=15mm{
				\mbox{$L(A)\underset{r}{\rtimes}\widehat{\mathbb{G}}\otimes B_0\underset{r}{\rtimes}\widehat{\mathbb{H}}$}\ar[d]_{\mbox{$\mathcal{Z}_{B_0}(u')\underset{r}{\rtimes}\widehat{\mathbb{F}}$}}\ar[r]^-{\mbox{$\underset{\sim}{\Psi_{B_0}}$}}&\mbox{$L(A\otimes B_0)\underset{r}{\rtimes}\widehat{\mathbb{F}}$}\ar[d]^{\mbox{$u\underset{r}{\rtimes} \widehat{\mathbb{F}}$}}\\
				\mbox{$A\underset{r}{\rtimes}\widehat{\mathbb{G}}\otimes B_0 \underset{r}{\rtimes}\widehat{\mathbb{H}}$}\ar[r]^{\mbox{$\cong$}}&\mbox{$(A\otimes B_0)\underset{r}{\rtimes}\widehat{\mathbb{F}}$}}
		\end{gathered}
		\end{equation*}
		is commutative where $\Psi_{B_0}:=\psi_{B_0}\underset{r}{\rtimes}\widehat{\mathbb{F}}$ and $u'$, $u$ are the Dirac homomorphism for $A$, $A\otimes B_0$, respectively
		\end{enumerate}
	\end{lem}
	\begin{proof}
		The proof of item $(i)$ follows the same argument as the one in item $(i)$ of \cite[Lemma 5.2.2]{RubenSemiDirect}. For item $(ii)$, since we assume that $\widehat{\mathbb{H}}$ satisfies the strong quantum BC property, one has ${_{A_0}}\mathcal{Z}(\mathscr{N}_{\widehat{\mathbb{H}}})\subset \mathscr{N}_{\widehat{\mathbb{F}}}$, for all $(A_0, \alpha_0)\in \text{Obj}(\mathscr{KK}^{\widehat{\mathbb{G}}})$ thanks to Lemma \ref{lem.FunctorPreservingLDirectProduct}. Then, the argument appearing in item $(ii)$ of \cite[Lemma 5.2.2]{RubenSemiDirect} can be applied verbatim. Finally, item $(iii)$ is the symmetric statement as item $(ii)$ by exchanging $\widehat{\mathbb{G}}$ with $\widehat{\mathbb{H}}$.
	\end{proof}

	 A special case of the previous lemma, which is important for our purpose, is the situation when $\mathbb{H}:=\mathbb{E}$ is the trivial quantum group. In this case, one has $\mathbb{F}=\mathbb{G}$ and so $(L, N)=(L', N')$. Given a $\widehat{\mathbb{G}}$-C$^*$-algebra $(A,\alpha)$, consider its $(\mathscr{L}_{\widehat{\mathbb{G}}}, \mathscr{N}_{\widehat{\mathbb{G}}})$-triangle in $\mathscr{KK}^{\widehat{\mathbb{G}}}$, say $\Sigma(N(A))\rightarrow L(A)\rightarrow A\rightarrow N(A)$. Next, given a C$^*$-algebra $B\in\text{Obj}(\mathscr{KK}^{\mathbb{E}})$ we apply the triangulated functor $\mathcal{Z}_{B}$ (cf. Lemma \ref{lem.FunctorPreservingLDirectProduct}) to this triangle and obtain the distinguished triangle $\Sigma(N(A)\otimes B)\rightarrow L(A)\otimes B\rightarrow A\otimes B\rightarrow N(A)\otimes B$ in $\mathscr{KK}^{\widehat{\mathbb{G}}}$.
	 
	 On the one hand, $\mathscr{KK}^{\mathbb{E}}=\mathscr{L}_{\mathbb{E}}$ (because the trivial quantum group satisfies the strong quantum BC property), hence $B\in\mathscr{L}_{\mathbb{E}}$ and so $L(A)\otimes B\in\mathscr{L}_{\widehat{\mathbb{G}}}$ (cf. Lemma \ref{lem.FunctorPreservingLDirectProduct}). On the other hand, one also has $N(A)\otimes B\in\mathscr{N}_{\widehat{\mathbb{G}}}$. This is true because if $\mathbb{H}:=\mathbb{E}$ is the trivial quantum group, then $\text{Tor}(\widehat{\mathbb{F}})=\text{Tor}(\widehat{\mathbb{G}})$ and $\mathcal{Z}(\widehat{\mathscr{N}}_{\widehat{\mathbb{G}}}\times \mathscr{KK})\subset \widehat{\mathscr{N}}_{\widehat{\mathbb{G}}}$. Indeed, if $(N, \delta)\in \text{Obj}(\widehat{\mathscr{N}}_{\widehat{\mathbb{G}}})$ and $B\in\text{Obj}(\mathscr{KK})$, then $j_{\mathbb{G}, T}(N\otimes B)=(N\otimes B)\underset{r, \delta\otimes id}{\rtimes}\mathbb{G}\underset{r, \overline{\tau}}{\ltimes}T^{op}\cong (N\underset{r, \delta}{\rtimes}\mathbb{G}\underset{r, \overline{\tau}}{\ltimes}T^{op})\otimes B\cong 0$ in $\mathscr{KK}$, for every torsion action $(T, \tau)$ of $\mathbb{G}$; where we have used Proposition \ref{pro.TwoSidedDirectProd}. In other words, $\mathcal{Z}(\widehat{\mathscr{N}}_{\widehat{\mathbb{G}}}\times \mathscr{KK})\subset \widehat{\mathscr{N}}_{\widehat{\mathbb{G}}}$, which yields $\mathcal{Z}(\mathscr{N}_{\widehat{\mathbb{G}}}\times \mathscr{KK})\subset \mathscr{N}_{\widehat{\mathbb{G}}}$ by virtue of Baaj-Skandalis duality. Hence $N(A)\otimes B\in\mathscr{N}_{\widehat{\mathbb{G}}}$ as claimed above.
	 
	 In conclusion, we have showed that the distinguished triangle $\Sigma(N(A)\otimes B)\rightarrow L(A)\otimes B\rightarrow A\otimes B\rightarrow N(A)\otimes B$ in $\mathscr{KK}^{\widehat{\mathbb{G}}}$ is a $(\mathscr{L}_{\widehat{\mathbb{G}}}, \mathscr{N}_{\widehat{\mathbb{G}}})$-triangle in $\mathscr{KK}^{\widehat{\mathbb{G}}}$ for $A\otimes B$. Finally, the uniqueness of this kind of triangles together with Proposition \ref{pro.DirectProductTensorProduct} yield the following: 
	 
	\begin{pro}\label{pro.InvertibleElementDirectProduct}
		Let $\mathbb{G}$ be a compact quantum group. For all $\widehat{\mathbb{G}}$-C$^*$-algebra $(A,\alpha)\in \text{Obj}(\mathscr{KK}^{\widehat{\mathbb{G}}})$ and all C$^*$-algebra $B\in \text{Obj}(\mathscr{KK})$ there exists an \emph{invertible} Kasparov triple $\psi_B\in KK^{\widehat{\mathbb{G}}}\big(L(A)\otimes B, L(A\otimes B)\big)$ such that the diagram: 
		\begin{equation}\label{eq.CommutativeDiagramDirectProductbisGNoH}
		\begin{gathered}
			\xymatrix@C=20mm@!R=15mm{
				\mbox{$ L(A)\underset{r}{\rtimes}\widehat{\mathbb{G}}\otimes B$}\ar[d]_{\mbox{$\mathcal{Z}_{B_0}(u')\underset{r}{\rtimes}\widehat{\mathbb{G}} $}}\ar[r]^-{\mbox{$\underset{\sim}{\Psi_{B}}$}}&\mbox{$ L(A\otimes B)\underset{r}{\rtimes}\widehat{\mathbb{G}}$}\ar[d]^{\mbox{$u\underset{r}{\rtimes} \widehat{\mathbb{G}}$}}\\
				\mbox{$ A\underset{r}{\rtimes}\widehat{\mathbb{G}}\otimes B$}\ar[r]^{\mbox{$\cong$}}&\mbox{$(A\otimes B)\underset{r}{\rtimes}\widehat{\mathbb{G}}$}}
		\end{gathered}
		\end{equation}
		is commutative where $\Psi_{B}:=\psi_{B}\underset{r}{\rtimes}\widehat{\mathbb{G}}$ and $u'$, $u$ are the Dirac homomorphism for $A$, $A\otimes B$, respectively.
	\end{pro}
	\begin{rem}\label{rem.AssumptionStrongBC}
		The previous corollary is a refined version of item $(iii)$ of Lemma \ref{lem.InvertibleElementDirectProduct}. The main observation in the case of $\mathbb{H}=\mathbb{E}$ is that we do not need to further assume that $\widehat{\mathbb{G}}$ satisfies the strong quantum BC property. As the argument above showed, this is possible because the torsion of $\widehat{\mathbb{F}}$ reduces to the torsion of $\widehat{\mathbb{G}}$ because $\mathbb{F}=\mathbb{G}$ as soon as $\mathbb{H}$ is the trivial quantum group.
		
		The additional hypothesis about the strong quantum BC property in Lemma \ref{lem.InvertibleElementDirectProduct} is necessary in order to use that the functor $\mathcal{Z}$ preserves the subcategories $\mathscr{N}_{*}$ as argued in Lemma \ref{lem.FunctorPreservingLDirectProduct}. This property holds automatically for classical locally compact groups (cf. \cite[Section 10.4]{MeyerNest}). In the quantum setting, the main obstacle is the understanding of the torsion phenomenon for a quantum direct product. Namely, as explained in Section \ref{sec.QuantumBC}, the subcategory $\widehat{\mathscr{N}}_{\widehat{\mathbb{F}}}$ is completely described by the vanishing of the functors $j_{\mathbb{F}, R}(\ \cdot\ )$ with $R\in \text{Tor}(\widehat{\mathbb{F}})$. It would be interesting to know the description of $\text{Tor}(\widehat{\mathbb{F}})$ in terms of $\text{Tor}(\widehat{\mathbb{G}})$ and $\text{Tor}(\widehat{\mathbb{H}})$ (as explained in the introduction, it is the subject of a subsequent paper) and then the behaviour of these C$^*$-algebras with respect to the two-sided crossed product construction, which is out of the scope of the present paper. This is in line with Proposition \ref{pro.InvertibleElementDirectProduct} above. 
		
		In general, the strong quantum BC property assumption allows to overlook this issue and in practice it is not too restrictive. On the one hand, all the examples of compact quantum groups studied in this context have duals satisfying the strong quantum BC property (cf. Section \ref{sec.KTheoryComp}). On the other hand, the permanence property established in Theorem \ref{theo.BCDirectProducts} (which is analogue to \cite[Theorem 5.3]{ChabertEchterhoffOyono}) requires an additional assumption with respect to the classical setting due to our approach to the equivariant Künneth class in the quantum setting (cf. Section \ref{sec.QKunnethClass}). This additional assumption is automatically fulfilled as soon as the strong quantum BC property holds (see Remark \ref{rem.FurtherHypoth2} for more details).
	\end{rem}
	
	Another consequence of Lemma \ref{lem.InvertibleElementDirectProduct} is that a quantum direct product satisfies the quantum BC property with coefficients in tensor products as soon as the quantum groups involved satisfy the strong quantum BC property. This result is the content of \cite[Theorem 5.2.3-(i)]{RubenSemiDirect} under torsion-freeness assumption. Here we are able to improve it by removing this assumption. The argument, as based on Lemma \ref{lem.InvertibleElementDirectProduct}, remains the same as the one in \cite[Theorem 5.2.3-(i)]{RubenSemiDirect} (see also  \cite[Remark 5.2.4]{RubenSemiDirect}). So, we have:
	\begin{theo}\label{theo.StrongBCDirectProd}
		Let $\mathbb{G}$, $\mathbb{H}$ be two compact quantum groups and let $\mathbb{F}:=\mathbb{G}\times \mathbb{H}$ be the corresponding quantum direct product of $\mathbb{G}$ and $\mathbb{H}$. If $\widehat{\mathbb{G}}$ and $\widehat{\mathbb{H}}$ satisfy the strong quantum BC property, then $\widehat{\mathbb{F}}$ satisfies the quantum BC property with coefficients in $A\otimes B$, for every $A\in \text{Obj}(\mathscr{KK}^{\widehat{\mathbb{G}}})$ and $B\in \text{Obj}(\mathscr{KK}^{\widehat{\mathbb{H}}})$. Moreover, $L(A\otimes B)\cong A\otimes B$, for every $A\in \text{Obj}(\mathscr{KK}^{\widehat{\mathbb{G}}})$ and $B\in \text{Obj}(\mathscr{KK}^{\widehat{\mathbb{H}}})$.
	\end{theo}
	\begin{rem}
		In the course of the proof of \cite[Theorem 5.2.3-(i)]{RubenSemiDirect} it is shown that $\eta^{\widehat{\mathbb{G}}}_A\otimes \eta^{\widehat{\mathbb{H}}}_B=\eta^{\widehat{\mathbb{F}}}_{A\otimes B}$ as soon as both $\widehat{\mathbb{G}}$ and $\widehat{\mathbb{H}}$ satisfy the strong quantum BC property.
	\end{rem}

	Note that the analogous assertion for the usual quantum BC property needs further hypothesis, which are related to the Künneth formula in order to compute the K-theory of a tensor product. This relation will be studied in Section \ref{sec.BCKunneth}.

	\subsection{The Künneth class $\mathcal{N}_{\widehat{\mathbb{G}}}$}\label{sec.QKunnethClass}
	The previous preliminary observations allow us to formulate an equivariant Künneth class in the spirit of Examples \ref{exs.KunnethClasses} and \cite{ChabertEchterhoffOyono} in the context of quantum groups. More precisely, let $\mathbb{G}$ be a compact quantum group, $(A, \alpha)$ a $\widehat{\mathbb{G}}$-C$^*$-algebra and $B$ a C$^*$algebra. The exterior Kasparov product $\widetilde{\tau}_\mathbb{C}$ together with the invertible Kasparov triple $\psi_B$ obtained in Proposition \ref{pro.InvertibleElementDirectProduct} give:
	$$KK(\mathbb{C},  L(A)\underset{r}{\rtimes} \widehat{\mathbb{G}})\times KK(\mathbb{C}, B)\overset{\widetilde{\tau}_{\mathbb{C}}}{\longrightarrow} KK(\mathbb{C}, L(A)\underset{r}{\rtimes} \widehat{\mathbb{G}}\otimes B)\cong KK(\mathbb{C}, L(A\otimes B)\underset{r}{\rtimes} \widehat{\mathbb{G}}),$$
	which yields a group homomorphism $\alpha_{\widehat{\mathbb{G}}}: K_*(L(A)\underset{r}{\rtimes} \widehat{\mathbb{G}})\otimes K_*(B)\rightarrow K_*(L(A\otimes B)\underset{r}{\rtimes} \widehat{\mathbb{G}})$, which is natural on $A$ and $B$.
	
	\begin{defi}
		Let $\mathbb{G}$ be a compact quantum group. We denote by $\mathcal{N}_{\widehat{\mathbb{G}}}$ the class of all $\widehat{\mathbb{G}}$-C$^*$-algebras $(A,\alpha)$ in $\widehat{\mathbb{G}}\text{-C}^*\text{-Alg}$ such that $\alpha_{\widehat{\mathbb{G}}}$ is an isomorphism for all separable C$^*$-algebra $B$ with $K_*(B)$ free abelian. By abuse of notation, we still denote by $\mathcal{N}_{\widehat{\mathbb{G}}}$ the full subcategory of $\widehat{\mathbb{G}}\text{-C}^*\text{-Alg}$ generated by the class $\mathcal{N}_{\widehat{\mathbb{G}}}$.
	\end{defi}
	
	Note that the homomorphism $\alpha_{\widehat{\mathbb{G}}}$ is obtained, up to the identification given by Proposition \ref{pro.InvertibleElementDirectProduct}, as it was obtained the group homomorphism $\alpha$ from Examples \ref{exs.KunnethClasses}. Here we use $L(A)\underset{r}{\rtimes} \widehat{\mathbb{G}}$ instead of an arbitrary C$^*$-algebra $A$. In this sense, we have that $(A,\alpha)\in \mathcal{N}_{\widehat{\mathbb{G}}}$ $\Leftrightarrow$ $L(A)\underset{r}{\rtimes} \widehat{\mathbb{G}}\in\mathcal{N}$, by definition. Therefore, $\mathcal{N}_{\widehat{\mathbb{G}}}$ is a Künneth category in the sense of Definition \ref{defi.KunnethCategory} and, given $(A,\alpha)\in \mathcal{N}_{\widehat{\mathbb{G}}}$, the map $B\mapsto K_*(L(A\otimes B)\underset{r}{\rtimes} \widehat{\mathbb{G}})$ defines a functor on the whole category $\text{C}^*\text{-Alg}$ satisfying properties $(i)-(iv)$ from Definition \ref{defi.KunnethFunctor}. In other words, this map defines a Künneth functor. Consequently, Theorem \ref{theo.AbstractKunnethTheorem} yields that the Künneth sequence:
	$$0\longrightarrow K_*(L(A)\underset{r}{\rtimes} \widehat{\mathbb{G}})\otimes K_*(B)\overset{\alpha_{\widehat{\mathbb{G}}}}{\longrightarrow} K_*(L(A\otimes B)\underset{r}{\rtimes} \widehat{\mathbb{G}})\overset{\beta_{\widehat{\mathbb{G}}}}{\longrightarrow} \text{Tor}(K_*(L(A)\underset{r}{\rtimes} \widehat{\mathbb{G}}), K_*(B))\longrightarrow 0$$
	is exact for all $(A,\alpha)\in\mathcal{N}_{\widehat{\mathbb{G}}}$ and all $B\in \text{Obj}(\text{C}^*\text{-Alg})$.
	
	\begin{ex}
		If $\widehat{\mathbb{G}}$ is a classical locally compact group $G$, then $\mathcal{N}_{\widehat{\mathbb{G}}}=\mathcal{N}_G$ as defined in Examples \ref{exs.KunnethClasses}. Indeed, from the Meyer-Nest reformulation of the BC property (see \cite[Theorem 5.2]{MeyerNest}) one knows that $K^{\text{top}}_*(G, A)\cong K_*(L(A)\underset{r}{\rtimes} G)$ naturally, for all $G$-C$^*$-algebra $(A,\alpha)$. Moreover, this approach to $\mathcal{N}_G$ yields a characterisation of the objects in the equivariant Künneth class in terms of the non-equivariant one, up to replacing $A$ by a $\mathscr{L}$-simplicial approximation of $A$. Namely, $(A,\alpha)\in \mathcal{N}_G$ $\Leftrightarrow$ $L(A)\underset{r}{\rtimes} G\in\mathcal{N}$.
	\end{ex}

\section{\textsc{Relating the quantum Baum-Connes property with the Künneth formula}}\label{sec.BCKunneth}
	In this section we generalise to the case of discrete quantum groups some key results appearing in \cite{ChabertEchterhoffOyono}. In particular, we are going to improve the stability of the quantum BC property for quantum direct products appearing already in \cite{RubenSemiDirect}.
	\begin{proSec}\label{pro.BCKunneth}
		Let $\mathbb{G}$ be a compact quantum group. The diagram:
		\begin{equation}\label{eq.CommutativeDiagramBCKunneth}
		\begin{gathered}
			\xymatrix@C=20mm@!R=15mm{
				\mbox{$ K_*(L(A)\underset{r}{\rtimes}\widehat{\mathbb{G}})\otimes K_*(B)$}\ar[d]_{\mbox{$\alpha_{\widehat{\mathbb{G}}}$}}\ar[r]^-{\mbox{$\eta^{\widehat{\mathbb{G}}}_{A}\otimes id$}}&\mbox{$ K_*(A\underset{r}{\rtimes}\widehat{\mathbb{G}})\otimes K_*(B)$}\ar[d]^{\mbox{$\alpha$}}\\
				\mbox{$ K_*(L(A\otimes B)\underset{r}{\rtimes}\widehat{\mathbb{G}})$}\ar[r]^{\mbox{$\eta^{\widehat{\mathbb{G}}}_{A\otimes B}$}}&\mbox{$K_*((A\otimes B)\underset{r}{\rtimes}\widehat{\mathbb{G}})$}}
		\end{gathered}
		\end{equation}
		commutes for all $\widehat{\mathbb{G}}$-C$^*$-algebra $A$ and all C$^*$-algebra $B$. In particular, if $\widehat{\mathbb{G}}$ satisfies the BC property with coefficients in $A\otimes B$, for all C$^*$-algebra $B$ equipped with the trivial action of $\widehat{\mathbb{G}}$; then $A\in\mathcal{N}_{\widehat{\mathbb{G}}}$ $\Leftrightarrow$ $A\underset{r}{\rtimes} \widehat{\mathbb{G}}\in\mathcal{N}$.
	\end{proSec}
	\begin{proof}
		Recall that the quantum assembly map $\eta^{\widehat{\mathbb{G}}}$ can be pictured as a certain Kasparov product (cf. Equation (\ref{eq.PictureQAssemblyMap})). Let $y\in K_*(L(A)\underset{r}{\rtimes}\widehat{\mathbb{G}})=KK(\mathbb{C}, L(A)\underset{r}{\rtimes}\widehat{\mathbb{G}})$ and $x\in K_*(B)=KK(\mathbb{C}, B)$. On the one hand, we have:
		\begin{equation*}
		\begin{split}
			y\otimes x\overset{\alpha_{\widehat{\mathbb{G}}}}{\mapsto} y\ \underset{\mathbb{C}}{\otimes^{\text{ext}}} x \overset{\eta^{\widehat{\mathbb{G}}}_{A\otimes B}}{\mapsto} (y\ \underset{\mathbb{C}}{\otimes^{\text{ext}}} x)\underset{L(A\otimes B)\underset{r}{\rtimes}\widehat{\mathbb{G}}}{\otimes} u\underset{r}{\rtimes} \widehat{\mathbb{G}}&\cong (y\ \underset{\mathbb{C}}{\otimes^{\text{ext}}} x)\underset{L(A)\underset{r}{\rtimes}\widehat{\mathbb{G}}\otimes B}{\otimes} \big(u'\underset{r}{\rtimes} \widehat{\mathbb{G}}\otimes id\big)\\
			&=\big(y \underset{L(A)\underset{r}{\rtimes}\widehat{\mathbb{G}}}{\otimes} u'\underset{r}{\rtimes} \widehat{\mathbb{G}}\big)\ \underset{\mathbb{C}}{\otimes^{\text{ext}}} \big(x\underset{B}{\otimes} id),
		\end{split}
		\end{equation*}
		where the last equality follows from Equation (\ref{eq.TensorProdKasparovProd}). On the other hand, we have:
		\begin{equation*}
		\begin{split}
			y\otimes x&\overset{\eta^{\widehat{\mathbb{G}}}_A\otimes id}{\mapsto} \big(y \underset{L(A)\underset{r}{\rtimes}\widehat{\mathbb{G}}}{\otimes} u'\underset{r}{\rtimes} \widehat{\mathbb{G}}\big)\otimes x\overset{\alpha}{\mapsto} \big(y \underset{L(A)\underset{r}{\rtimes}\widehat{\mathbb{G}}}{\otimes} u'\underset{r}{\rtimes} \widehat{\mathbb{G}}\big)\ \underset{\mathbb{C}}{\otimes^{\text{ext}}} x.
		\end{split}
		\end{equation*}
		
		These computations show that, indeed, Diagram (\ref{eq.CommutativeDiagramBCKunneth}) commutes. To conclude, notice that if $\widehat{\mathbb{G}}$ satisfies the BC property with coefficients in $A\otimes B$, for all C$^*$-algebra $B$ equipped with the trivial action of $\widehat{\mathbb{G}}$; then both $\eta_A$ and $\eta_{A\otimes B}$ are isomorphisms, for all C$^*$-algebra $B$. Therefore, commutativity of Diagram (\ref{eq.CommutativeDiagramBCKunneth}) yields that $\alpha_{\widehat{\mathbb{G}}}$ is an isomorphism if and only if $\alpha$ is an isomorphism.
	\end{proof}
	
	\begin{proSec}\label{pro.KunnethBC}
		Let $\mathbb{G}$ be a compact quantum group. Assume that $A$ is a $\widehat{\mathbb{G}}$-C$^*$-algebra such that $\widehat{\mathbb{G}}$ satisfies the quantum BC property with coefficients in $A$. 
		\begin{enumerate}[i)]
			\item If $A\in\mathcal{N}_{\widehat{\mathbb{G}}}$ and $A\underset{r}{\rtimes} \widehat{\mathbb{G}}\in\mathcal{N}$, then $\widehat{\mathbb{G}}$ satisfies the quantum BC property for $A\otimes B$, for all C$^*$-algebra $B$.
			\item If $A\in\mathcal{N}_{\widehat{\mathbb{G}}}$, then $\widehat{\mathbb{G}}$ satisfies the quantum BC property for $A\otimes B$, for all C$^*$-algebra $B\in\mathcal{N}$.
		\end{enumerate}
	\end{proSec}
	\begin{proof}
	First, recall that given a $\widehat{\mathbb{G}}$-C$^*$-algebra $A$, we have $A\in \mathcal{N}_{\widehat{\mathbb{G}}}$ $\Leftrightarrow$ $L(A)\underset{r}{\rtimes} \widehat{\mathbb{G}}\in\mathcal{N}$ by definition. Next, consider the following diagram:
	$$
	\xymatrix@C=15mm@!R=18mm{
		\mbox{$0\longrightarrow K_*(L(A)\underset{r}{\rtimes}\widehat{\mathbb{G}})\otimes K_*(B)$}\ar[r]^-{\mbox{$\alpha_{\widehat{\mathbb{G}}}$}}\ar[d]_{\mbox{$\eta^{\widehat{\mathbb{G}}}_{A}\otimes id$}}
		&\mbox{$K_*(L(A\otimes B)\underset{r}{\rtimes}\widehat{\mathbb{G}})$}\ar[r]^-{\mbox{$\beta_{\widehat{\mathbb{G}}}$}}\ar[d]^{\mbox{$\eta^{\widehat{\mathbb{G}}}_{A\otimes B}$}}
		&\mbox{$\text{Tor}(K_*(L(A)\underset{r}{\rtimes} \widehat{\mathbb{G}}), K_*(B))\longrightarrow 0$}\ar[d]^{\mbox{$\text{Tor}(\eta^{\widehat{\mathbb{G}}}_{A}\otimes id)$}}\\
		\mbox{$0\longrightarrow K_*(A\underset{r}{\rtimes}\widehat{\mathbb{G}})\otimes K_*(B)$}\ar[r]_-{\mbox{$\alpha$}}
		&\mbox{$K_*((A\otimes B)\underset{r}{\rtimes}\widehat{\mathbb{G}})$}\ar[r]_-{\mbox{$\beta$}}
		&\mbox{$\text{Tor}(K_*(A\underset{r}{\rtimes}\widehat{\mathbb{G}}), K_*(B)) \longrightarrow 0$}}
	$$
	
	The argument follows then the same lines as the one in \cite[Proposition 4.10]{ChabertEchterhoffOyono}.
	\end{proof}
	
	\begin{corSec}\label{cor.KunnethBC2}
		Let $\mathbb{G}$ be a compact quantum group. Assume that $\widehat{\mathbb{G}}$ satisfies the quantum BC property with coefficients in $\mathbb{C}$.
		\begin{enumerate}[i)]
			\item If $C(\mathbb{G})\in\mathcal{N}$ and $\mathbb{C}\in\mathcal{N}_{\widehat{\mathbb{G}}}$, then $\widehat{\mathbb{G}}$ satisfies the quantum BC property for all C$^*$-algebra $B$ equipped with the trivial action of $\widehat{\mathbb{G}}$.
			\item If $\mathbb{C}\in\mathcal{N}_{\widehat{\mathbb{G}}}$, then $\widehat{\mathbb{G}}$ satisfies the quantum BC property for all C$^*$-algebra $B\in\mathcal{N}$.
		\end{enumerate}
	\end{corSec}
	\begin{proof}
		Recall that $C(\mathbb{G})=\mathbb{C}\underset{r}{\rtimes} \widehat{\mathbb{G}}$. So, $(i)$ and $(ii)$ of statement follow from $(i)$ and $(ii)$ of Proposition \ref{pro.KunnethBC}, respectively.	
	\end{proof}
	
		\begin{remSec}\label{rem.FurtherHypoth}
	In the classical setting when $\widehat{\mathbb{G}}$ is a locally compact group $G$, the assumption $\mathbb{C}\in\mathcal{N}_G$ is automatically satisfied. Indeed, $\mathbb{C}$ is a type I C$^*$-algebra and $\mathcal{N}_G$ contains all type I $G$-C$^*$-algebras by virtue of \cite[Theorem 0.1]{ChabertEchterhoffOyono}. In the quantum setting, a similar related result to \cite[Theorem 0.1]{ChabertEchterhoffOyono} is Theorem \ref{theo.QuantumGroupCalgebraKunneth} in the next section (cf. Remark \ref{rem.AnalogoyThm01} for an explanation to this analogy). However, to the best knowledge of the author, it is not known for instance whether $\mathbb{C}\in\mathcal{N}_{\widehat{\mathbb{G}}}$ for every discrete quantum group $\widehat{\mathbb{G}}$. One reason for this is that in our approach the objects in $\mathcal{N}_{\widehat{\mathbb{G}}}$ are characterised in terms of objects in $\mathcal{N}$ up to a \emph{$\mathscr{L}_{\widehat{\mathbb{G}}}$-simplicial approximation}, which entails to study the localisation functor $L$ in relation with crossed products and the \emph{equivariant} Künneth class. One possibility to do so might be to adapt the \emph{Going-Down technique} from \cite{ChabertEchterhoffOyono} based on Theorem \ref{theo.QuantumGroupCalgebraKunneth}. But this is out of the scope of the present article.
	\end{remSec}
	
	The following theorem is an improved version of \cite[Corollary 5.2.5]{RubenSemiDirect}.
	\begin{theoSec}\label{theo.BCDirectProducts}
		Let $\mathbb{G}$, $\mathbb{H}$ be two compact quantum groups and let $\mathbb{F}:=\mathbb{G}\times \mathbb{H}$ be the corresponding quantum direct product of $\mathbb{G}$ and $\mathbb{H}$.
		\begin{enumerate}[i)]
			\item Let $A$ be a $\widehat{\mathbb{G}}$-C$^*$-algebra and $B$ a $\widehat{\mathbb{H}}$-C$^*$-algebra. Assume that $\widehat{\mathbb{G}}$ satisfies the strong quantum BC property and that $B\in\mathscr{L}_{\widehat{\mathbb{H}}}$. If $A\underset{r}{\rtimes}\widehat{\mathbb{G}}\in\mathcal{N}$ and $A\in\mathcal{N}_{\widehat{\mathbb{G}}}$, then $\widehat{\mathbb{F}}$ satisfies the quantum BC property with coefficients in $A\otimes B$.
			
			In particular, if $\widehat{\mathbb{G}}$ satisfies the strong quantum BC property, $C(\mathbb{G})\in\mathcal{N}$ and $\mathbb{C}\in\mathscr{L}_{\widehat{\mathbb{H}}}$; then $\widehat{\mathbb{F}}$ satisfies the quantum BC property with coefficients in $\mathbb{C}$.
			
			\item Assume that $\widehat{\mathbb{G}}$ and $\widehat{\mathbb{H}}$ satisfy the strong quantum BC property. If $C(\mathbb{G}), C(\mathbb{H})\in\mathcal{N}$ and $\mathbb{C}\in\mathcal{N}_{\widehat{\mathbb{F}}}$, then $\widehat{\mathbb{F}}$ satisfies the quantum BC property for all C$^*$-algebra $B$ equipped with the trivial action of $\widehat{\mathbb{F}}$.
		\end{enumerate}
	\end{theoSec}
	\begin{proof}
		\begin{enumerate}[i)]
			\item Consider the following diagram:
			$$
	\xymatrix@C=5mm@!R=15mm{
		\mbox{$0$}\ar[r]&\mbox{$ K_*(L'(A)\underset{r}{\rtimes}\widehat{\mathbb{G}})\otimes K_*(L''(B)\underset{r}{\rtimes}\widehat{\mathbb{H}})$}\ar[r]^-{\mbox{$\alpha_{\widehat{\mathbb{G}}}$}}\ar[d]_{\mbox{$\eta^{\widehat{\mathbb{G}}}_{A}\otimes \eta^{\widehat{\mathbb{H}}}_{B}$}}
		&\mbox{$K_*(L(A\otimes B)\underset{r}{\rtimes}\widehat{\mathbb{F}})$}\ar[r]^-{\mbox{$\beta_{\widehat{\mathbb{G}}}$}}\ar[d]^{\mbox{$\eta^{\widehat{\mathbb{F}}}_{A\otimes B}$}}
		&\mbox{$\text{Tor}(K_*(L'(A)\underset{r}{\rtimes} \widehat{\mathbb{G}}), K_*(L''(B)\underset{r}{\rtimes}\widehat{\mathbb{H}}))\longrightarrow 0$}\ar[d]^{\mbox{$\text{Tor}(\eta^{\widehat{\mathbb{G}}}_{A}\otimes \eta^{\widehat{\mathbb{H}}}_{B}$)}}\\
		\mbox{$0$}\ar[r]&\mbox{$ K_*(A\underset{r}{\rtimes}\widehat{\mathbb{G}})\otimes K_*(B\underset{r}{\rtimes}\widehat{\mathbb{H}})$}\ar[r]_-{\mbox{$\alpha$}}
		&\mbox{$K_*((A\otimes B)\underset{r}{\rtimes}\widehat{\mathbb{F}})$}\ar[r]_-{\mbox{$\beta$}}
		&\mbox{$\text{Tor}(K_*(A\underset{r}{\rtimes}\widehat{\mathbb{G}}), K_*(B\underset{r}{\rtimes}\widehat{\mathbb{H}})) \longrightarrow 0$}}
	$$
			
	On the one hand, since $\widehat{\mathbb{G}}$ is assumed to satisfy the strong quantum BC property and $B\in\mathscr{L}_{\widehat{\mathbb{H}}}$, we can apply Lemma \ref{lem.InvertibleElementDirectProduct}. In particular, one has that $L'(A)\otimes L''(B)\cong L(A\otimes B)$ and the middle arrow is indeed $\eta^{\widehat{\mathbb{F}}}_{A\otimes B}$. On the other hand, $\eta^{\widehat{\mathbb{H}}}_{B}$ is an isomorphism (because  $B\cong L''(B)$) and $\eta^{\widehat{\mathbb{G}}}_A$ is an isomorphism (because $\widehat{\mathbb{G}}$ satisfies the strong quantum BC property). Consequently, the left arrow is an isomorphism hence so is the right arrow. Finally, the assumptions $A\underset{r}{\rtimes}\widehat{\mathbb{G}}\in\mathcal{N}$ and $A\in\mathcal{N}_{\widehat{\mathbb{G}}}$ say that the bottom and top lines of the diagram are exact sequences, respectively. In conclusion, the Five lemma yields that the middle arrow, i.e. $\eta^{\widehat{\mathbb{F}}}_{A\otimes B}$, is an isomorphism.

			\item Since $\widehat{\mathbb{G}}$ and $\widehat{\mathbb{H}}$ satisfy the strong quantum BC property by assumption, then Theorem \ref{theo.StrongBCDirectProd} yields that $\widehat{\mathbb{F}}$ satisfies the quantum BC property with coefficients in $A\otimes B$, for every $A\in \text{Obj}(\mathscr{KK}^{\widehat{\mathbb{G}}})$ and $B\in \text{Obj}(\mathscr{KK}^{\widehat{\mathbb{H}}})$. In particular, $\widehat{\mathbb{F}}$ satisfies the quantum BC property with coefficients in $\mathbb{C}$. Next, observe that $C(\mathbb{F})=C(\mathbb{G})\otimes C(\mathbb{H})\in\mathcal{N}$ because $\mathcal{N}$ is closed under taking tensor products (cf. Lemma \ref{lem.StabilityClassN}). Then the conclusion follows from item $(i)$ of Corollary \ref{cor.KunnethBC2}.
		\end{enumerate}
	\end{proof}
	
	\begin{remSec}\label{rem.FurtherHypoth2}
		Theorem \ref{theo.BCDirectProducts} above is analogue to \cite[Theorem 5.3]{ChabertEchterhoffOyono}. In the quantum setting we need however further hypothesis. Namely, we need $\mathbb{C}\in\mathcal{N}_{\widehat{\mathbb{G}}}$ as explained in Remark \ref{rem.FurtherHypoth} and $\mathbb{C}\in\mathscr{L}_{\widehat{\mathbb{H}}}$. The reason for the latter is again that we need some control for the $\mathscr{L}_{\widehat{\mathbb{H}}}$-simplicial approximations in relation with tensor products. This is made precise by Lemma \ref{lem.InvertibleElementDirectProduct} (and Remark \ref{rem.AssumptionStrongBC}), which we use in the proof of Theorem \ref{theo.BCDirectProducts}. In the classical setting, this control is guaranteed by \cite[Corollary 2.10]{ChabertEchterhoffOyono} based on their \emph{Going-Down technique}.
		
		Note that assuming that $\widehat{\mathbb{G}}$ satisfies the \emph{strong} quantum BC property guarantees the property $\mathbb{C}\in\mathcal{N}_{\widehat{\mathbb{G}}}$. Indeed, in this case one has in particular that $\mathbb{C}\in\mathscr{L}_{\widehat{\mathbb{G}}}$ hence $\mathbb{C}\cong L'(\mathbb{C})$. Therefore, by definition of our equivariant Künneth class, one has $\mathbb{C}\in\mathcal{N}_{\widehat{\mathbb{G}}}$ $\Leftrightarrow$ $C(\mathbb{G})\in\mathcal{N}$. The latter is our assumption in item $(i)$ of Theorem \ref{theo.BCDirectProducts}, which is the analogous assumption made in item $(i)$ of \cite[Theorem 5.3]{ChabertEchterhoffOyono}.
	\end{remSec}
	
\section{\textsc{Some K-theory computations}}\label{sec.KTheoryComp}

Recall that by Theorem \ref{theo.StrongBCDirectProd}, we have that if $\widehat{\mathbb{G}}$ and $\widehat{\mathbb{H}}$ satisfy the strong quantum BC property, then $\widehat{\mathbb{G}\times \mathbb{H}}$ satisfies the strong quantum BC property with coefficients in $A\otimes B$, for all $A\in \text{Obj}(\mathscr{KK}^{\widehat{\mathbb{G}}})$ and $B\in \text{Obj}(\mathscr{KK}^{\widehat{\mathbb{H}}})$. In order to compute the K-theory groups of $C(\mathbb{G}\times\mathbb{H})$ by applying the homological techniques from the Meyer-Nest work, we would need further information concerning the torsion phenomenon of $\widehat{\mathbb{F}}$. Namely, a classification of torsion actions of a quantum direct product, which seems hard in general as explained in the introduction. If such a classification is provided, then it is plausible to construct explicit projective resolutions for $\mathbb{C}$ in $\mathscr{KK}^{\widehat{\mathbb{G}\times \mathbb{H}}}$ as tensor products of projective resolutions for $\mathbb{C}$ in $\mathscr{KK}^{\widehat{\mathbb{G}}}$ and in $\mathscr{KK}^{\widehat{\mathbb{H}}}$. Note that, since $C(\mathbb{G}\times\mathbb{H})=C(\mathbb{G})\otimes C(\mathbb{H})$, we need either $C(\mathbb{G})$ or $C(\mathbb{H})$ to satisfy the Künneth formula to succeed in such a construction.

However, instead of doing that, we can compute $K_*(C(\mathbb{G}\times\mathbb{H}))$ using simply the Künneth formula. In order to do so, let us point out that \cite[Theorem 5.2]{YukiBCTorsion} and \cite[Corollary 5.5]{YukiBCTorsion} can be also obtained for the Künneth class instead of the bootstrap class. This is true because even if these two classes are not the same, they satisfy similar stabilising properties in the sense of Lemma \ref{lem.StabilityClassN}. Namely, $\mathcal{N}$ contains all finite dimensional C$^*$-algebras, it is closed under tensor products, it is closed under semi-split extensions (hence under mapping cones and homotopy limits) and it contains all type I C$^*$-algebras. In other words, the arguments in \cite[Theorem 5.2]{YukiBCTorsion} and \cite[Corollary 5.5]{YukiBCTorsion} can be applied \emph{verbatim} by replacing the bootstrap class by $\mathcal{N}$ and we obtain the following:
\begin{theoSec}\label{theo.QuantumGroupCalgebraKunneth}
	Let $\mathbb{G}$ be a compact quantum group. 
	\begin{enumerate}[i)]
		\item If $(B, \beta)$ is a $\mathbb{G}$-C$^*$-algebra such that $B\underset{r, \beta}{\rtimes}\mathbb{G}\underset{r, \overline{\delta}}{\ltimes}T^{op}\in\mathcal{N}$, for all torsion action $(T, \delta)\in\text{Tor}(\widehat{\mathbb{G}})$; then $\widehat{L}(B)\in\mathcal{N}$.
		\item Assume that $\widehat{\mathbb{G}}$ satisfies the strong quantum BC property. If $(A, \alpha)$ is a type $I$ $\widehat{\mathbb{G}}$-C$^*$-algebra, then $A\underset{r, \alpha}{\rtimes}\widehat{\mathbb{G}}\in\mathcal{N}$. In particular, $C(\mathbb{G})\in\mathcal{N}$.
	\end{enumerate}
\end{theoSec}

\begin{remSec}\label{rem.AnalogoyThm01}
	On the one hand, as we have explained in the introduction, the family of finite subgroups represents the torsion phenomenon for a classical discrete group. In the quantum setting it must be replaced by the family of torsion actions of a compact quantum group. This makes a substantial difference when it comes to define the quantum assembly map for a discrete quantum group. In particular, the induction functor must be replaced by the two-sided crossed product functor as explained in Section \ref{sec.QuantumBC}. In particular, we work in the category $\mathscr{KK}^{\mathbb{G}}$ and not in $\mathscr{KK}^{\widehat{\mathbb{G}}}$ itself. If $G$ is a discrete group, then \cite[Theorem 0.1]{ChabertEchterhoffOyono} gives a sufficient condition for a $G$-C$^*$-algebra to be in $\mathcal{N}_G$ in terms of the torsion of $G$. In this sense, item $(i)$ of Theorem \ref{theo.QuantumGroupCalgebraKunneth} gives a sufficient condition for the \emph{$\widehat{\mathscr{L}}_{\widehat{\mathbb{G}}}$-simplicial approximation} of a $\mathbb{G}$-C$^*$-algebra to be in $\mathcal{N}$ in terms of the torsion actions of $\mathbb{G}$. The fact that we work in $\mathscr{KK}^{\mathbb{G}}$ leads to consider \emph{dual} (in the sense of the Baaj-Skandalis duality) simplicial approximations, which would lead to a \emph{dual} equivariant Künneth class. However, the Künneth class $\mathcal{N}$ is not stable under general crossed products hence it is not clear whether item $(i)$ of Theorem \ref{theo.QuantumGroupCalgebraKunneth} translates into a statement about $\mathcal{N}_{\widehat{\mathbb{G}}}$.
	
	On the other hand, the conclusion of item $(ii)$ of Theorem \ref{theo.QuantumGroupCalgebraKunneth}, i.e. that $C(\mathbb{G})\in\mathcal{N}$ as soons as $\widehat{\mathbb{G}}$ satisfies the strong quantum BC property is also true for classical locally compact groups. One can argue as follows. Assume that $G$ is a locally compact group satisfying the BC property with coefficients (\emph{a fortriori} when $G$ satisfies the \emph{strong} BC property). As explained in Remark \ref{rem.FurtherHypoth}, we always have $\mathbb{C}\in\mathcal{N}_G$. Therefore, \cite[Proposition 4.9]{ChabertEchterhoffOyono} implies that $C^*(G)=\mathbb{C}\underset{r}{\rtimes} G\in\mathcal{N}$.
\end{remSec}

By the work of Voigt and Vergnioux-Voigt (e.g. \cite{VoigtBaumConnesFree}, \cite{VoigtBaumConnesUnitaryFree}, \cite{VoigtBaumConnesAutomorphisms}) we have a number of examples of compact quantum groups with duals satisfying the strong quantum BC property. Hence the C$^*$-algebras defining these compact quantum groups lie in $\mathcal{N}$. Namely:
\begin{corSec}\label{cor.QGKunnetClass}
	Let $\mathbb{G}$ be any of the following compact quantum groups: $SU_q(2)$, $O^+(F)$, $U^+(Q)$, $S^+_N$, where $N\in\mathbb{N}$, $Q$ is a complex invertible matrix and $F$ is a complex invertible matrix such that $F\overline{F}\in \mathbb{R} id$. Then $C(\mathbb{G})\in\mathcal{N}$.
\end{corSec}

Moreover, a computation of the K-theory groups of the C$^*$-algebras $C(\mathbb{G})$ is carried out too in the works \cite{VoigtBaumConnesFree}, \cite{VoigtBaumConnesUnitaryFree}, \cite{VoigtBaumConnesAutomorphisms}. Namely:
\begin{theoSec}\label{theo.KtheoryComp}
	Let $N\in\mathbb{N}$, let $Q$ be a complex invertible matrix and let $F$ be a complex invertible matrix such that $F\overline{F}\in \mathbb{R} id$. Then:
	\begin{itemize}[-]
		\item $K_0(C(SU_q(2)))=\mathbb{Z}$ and $K_1(C(SU_q(2)))=\mathbb{Z}$.
		\item $K_0(C(O^+(F)))=\mathbb{Z}$ and $K_1(C(O^+(F)))=\mathbb{Z}$.
		\item $K_0(C(U^+(Q)))=\mathbb{Z}$ and $K_1(C(U^+(Q)))=\mathbb{Z}\oplus\mathbb{Z}$.
		\item $K_0(C(S^+_N))=\mathbb{Z}^{N^2-2N+2}$ and $K_1(C(S^+_N))=\mathbb{Z}$.
	\end{itemize}
\end{theoSec}

These computations together with Theorem \ref{theo.QuantumGroupCalgebraKunneth} allow to compute the K-theory groups for the corresponding quantum direct products by means of the Künneth formula. More precisely, if $\mathbb{G}$ and $\mathbb{H}$ are any of the compact quantum groups $SU_q(2)$, $O^+(F)$, $U^+(Q)$ or $S^+_N$ as above, then it follows from Theorem \ref{theo.KtheoryComp} that $K_*(C(\mathbb{H}))$ is free abelian. Moreover, $C(\mathbb{G})\in\mathcal{N}$ by Corollary \ref{cor.QGKunnetClass}. Therefore, we have $K_*(C(\mathbb{F}))=K_*(C(\mathbb{G})\otimes C(\mathbb{H}))=K_*(C(\mathbb{G}))\otimes K_*(C(\mathbb{H}))$, where $\mathbb{F}=\mathbb{G}\times \mathbb{H}$. Recall that $K_*(\ \cdot\ )$ denotes the $\mathbb{Z}/2$-graded $K$-theory hence $K_l(C(\mathbb{F}))=\underset{\underset{i,j\in\{0,1\}}{i+j=l}}{\bigoplus}K_i(C(\mathbb{G}))\otimes K_j(C(\mathbb{H}))$, for all $l\in\{0,1\}$. Namely, we have the following:
\begin{theoSec}
	Let $N\in\mathbb{N}$, let $Q$ be a complex invertible matrix and let $F$ be a complex invertible matrix such that $F\overline{F}\in \mathbb{R} id$. Then:
	\begin{itemize}[-]
		\item For $\mathbb{F}:=O^+(F)\times O^+(F)$, we have $K_0(C(\mathbb{F}))=\mathbb{Z}^2$ and $K_1(C(\mathbb{F}))=\mathbb{Z}^2$.
		\item For $\mathbb{F}:=O^+(F)\times U^+(Q)$, we have $K_0(C(\mathbb{F}))=\mathbb{Z}^3$ and $K_1(C(\mathbb{F}))=\mathbb{Z}^3$.
		\item For $\mathbb{F}:=O^+(F)\times S^+_N$, we have $K_0(C(\mathbb{F}))=\mathbb{Z}^{N^2-2N+3}$ and $K_1(C(\mathbb{F}))=\mathbb{Z}^{N^2-2N+3}$.
		\item For $\mathbb{F}:=U^+(Q)\times U^+(Q)$, we have $K_0(C(\mathbb{F}))=\mathbb{Z}^5$ and $K_1(C(\mathbb{F}))=\mathbb{Z}^4$.
		\item For $\mathbb{F}:=U^+(Q)\times S^+_N$, we have $K_0(C(\mathbb{F}))=\mathbb{Z}^{N^2-2N+4}$ and $K_1(C(\mathbb{F}))=\mathbb{Z}^{2N^2-4N+5}$.
		\item For $\mathbb{F}:=S^+_N\times S^+_N$, we have $K_0(C(\mathbb{F}))=\mathbb{Z}^{(N^2-2N+2)^2+1}$ and $K_1(C(\mathbb{F}))=\mathbb{Z}^{2(N^2-2N+2)}$.
		
	\end{itemize}
\end{theoSec}

	Of course, a similar list of K-theory groups can be obtained with other combinations of compact quantum groups into a quantum direct product as soon as the strong quantum BC property is satisfied and the K-groups of the C$^*$-algebras defining the quantum groups involved are free abelian. Notice that, in order to apply the Künneth formula, the computation of the K-groups of the C$^*$-algebra defining $\mathbb{G}\times \mathbb{H}$ requires at least $K_*(C(\mathbb{H}))$ to be free abelian.
	
	For instance, let $\mathbb{G}$ be a compact quantum group, $N\geq 4$ and $\mathbb{G}\wr_*S^+_N$ the corresponding free wreath product of $\mathbb{G}$ by $S^+_N$. A result by F. Lemeux and P. Tarrago (cf. \cite{TarragoWreath}) shows that there exists a parameter $q\in [-1,1]$ such that the compact quantum group $\mathbb{H}_q$ is monoidal equivalent to $\mathbb{G}\wr_*S^+_N$, where $\mathbb{H}_q$ is such that $\widehat{\mathbb{H}}_q$ is the discrete quantum subgroup of $\widehat{\mathbb{G} * SU_q(2)}$ generated, in the sense of the Tannaka-Krein duality, by the representations $xux$ (as a word in $\text{Irr}(\mathbb{G} * SU_q(2))$) with $x\in\text{Irr}(\mathbb{G})$ and $u$ being the fundamental representation of $SU_q(2)$. It is shown in \cite{RubenAmauryTorsion} that $\widehat{\mathbb{G}\wr_*S^+_N}$ (hence $\widehat{\mathbb{H}}_q$) satisfies the strong quantum BC property as soon as $\widehat{\mathbb{G}}$ is torsion-free and satisfies the strong quantum BC property. 
	
	The K-theory of $C(\mathbb{H}_q)$ is computed in \cite{RubenAmauryTorsion} for concrete and relevant instances of $\mathbb{G}$. For example, when $\widehat{\mathbb{G}}:=\mathbb{F}_n$ is the classical free group with $n\in\mathbb{N}$ generators, then $K_0(C(\mathbb{H}_q))=\mathbb{Z}\oplus \mathbb{Z}_2$ and $K_1(C(\mathbb{H}_q))=\mathbb{Z}^{n+1}$. In this case, we see that the K-groups are \emph{not} free abelian, so that the Künneth formula cannot be applied to compute the K-theory of quantum direct products of the form $\mathbb{X}\times \mathbb{H}_q$, but it still applies for quantum direct products of the form $\mathbb{H}_q \times \mathbb{X}$ for $\mathbb{X}$ any compact quantum group satisfying the strong quantum BC property and such that $K_*(C(\mathbb{X}))$ is free abelian.
	
	The K-theory of $C(\mathbb{G}\wr_*S^+_N)$ is computed in the recent paper \cite{FimaTroupelCouronne} by P. Fima and A. Troupel for concrete and relevant instances of $\mathbb{G}$. For example, when $\widehat{\mathbb{G}}:=\mathbb{F}_n$ is the classical free group with $n\in\mathbb{N}$ generators, then $K_0(C(\widehat{\mathbb{F}}_n\wr_*S^+_N))=\mathbb{Z}^{N^2-2N+2}$ and $K_1(C(\widehat{\mathbb{F}}_n\wr_*S^+_N))=\mathbb{Z}^{N^2n+1}$ (cf. \cite[Corollary 7.2]{FimaTroupelCouronne}). In this case, the K-groups are free abelian.

\bibliographystyle{acm}
\bibliography{KunnethClassQuantumGroups}

\vspace{1cm}
\textsc{R. Martos, Department of Mathematical Sciences, University of Copenhagen, Denmark.} 

\textit{E-mail address:} \textbf{\texttt{ruben.martos@math.ku.dk}}

 \end{document}